\documentclass[12pt]{article}
\usepackage[a4paper,margin=2.5cm]{geometry}

\usepackage{amsmath}
\usepackage{amsthm}
\usepackage{amssymb}

\usepackage{tikz}
\usetikzlibrary{arrows.meta, decorations.markings, calc}

\tikzset{
  mid arrow/.style={
    postaction={decorate,
      decoration={markings,
        mark=at position #1 with {\arrow[scale=1.3]{Stealth}}
      }
    }
  },
  mid arrow/.default=0.55
}

\usepackage[labelsep=newline]{caption}
\usepackage{subcaption}

\usepackage[hidelinks]{hyperref}
\hypersetup{
  pdftitle={Whitham Deformations of Nodal Spectral Curves of the Sinh-Gordon Equation},
  pdfauthor={Ross Ogilvie},
  pdfsubject={Mathematics Subject Classification (2020): 14H15, 14H70, 37K20, 53A10, 30F30, 14H20},
  pdfkeywords={spectral curves, Whitham deformations, nodal curves, isoperiodic sets, sinh-Gordon equation, Willmore energy},
}

\newcommand{\labelthis}[1]{\stepcounter{equation}\tag{\theequation}\label{#1}}

\numberwithin{equation}{section}

\newtheorem{theorem}{Theorem}[section]
\newtheorem{corollary}[theorem]{Corollary}
\newtheorem{lemma}[theorem]{Lemma}

\theoremstyle{definition}
\newtheorem{definition}[theorem]{Definition}
\newtheorem{remark}[theorem]{Remark}
\newtheorem{example}[theorem]{Example}

\DeclareMathOperator{\Real}{Re}
\DeclareMathOperator{\Imag}{Im}
\DeclareMathOperator*{\Res}{Res}

\DeclareMathOperator{\Hom}{Hom}

\DeclareMathOperator{\GL}{GL}

\DeclareMathOperator{\Perm}{Perm}

\DeclareMathOperator{\Gr}{Gr}
\newcommand{\CP}{\bbC\bbP}
\newcommand{\RP}{\bbR\bbP}

\DeclareMathOperator{\Per}{Per}
\DeclareMathOperator{\APer}{APer}
\DeclareMathOperator{\BPer}{BPer}
\DeclareMathOperator{\CA}{CA}
\newcommand{\willmore}{\calW}

\newcommand{\ci}{i}
\newcommand{\dd}{\mathrm{d}}

\newcommand{\bbC}{\mathbb{C}}

\newcommand{\bbF}{\mathbb{F}}

\newcommand{\bbN}{\mathbb{N}}

\newcommand{\bbP}{\mathbb{P}}
\newcommand{\bbQ}{\mathbb{Q}}
\newcommand{\bbR}{\mathbb{R}}
\newcommand{\bbS}{\mathbb{S}}

\newcommand{\bbZ}{\mathbb{Z}}

\newcommand{\calB}{\mathcal{B}}

\newcommand{\calF}{\mathcal{F}}

\newcommand{\calH}{\mathcal{H}}

\newcommand{\calN}{\mathcal{N}}

\newcommand{\calP}{\mathcal{P}}

\newcommand{\calS}{\mathcal{S}}

\newcommand{\calU}{\mathcal{U}}

\newcommand{\calW}{\mathcal{W}}

\parskip = \baselineskip

\newcommand{\refinv}[2]{#2 of~\cite{Hasse2026a}}
\newcommand{\eqrefinv}[2]{(#2) of~\cite{Hasse2026a}}
\newcommand{\numinv}[2]{#2}

\title{Whitham Deformations of Nodal Spectral Curves of the Sinh-Gordon Equation}
\author{Ross Ogilvie}
\date{}

\begin{document}
\maketitle
\begin{NoHyper}\renewcommand{\thefootnote}{}\footnotetext{Date: \today.}\end{NoHyper}
\begin{abstract}
Spectral curves of constant mean curvature tori in Euclidean three-space are real hyperelliptic curves carrying two meromorphic differentials with integral periods.
Deformations that preserve the periods, called Whitham deformations, are well understood for nonsingular curves, where the isoperiodic sets are manifolds.
We extend Whitham deformations to curves with ordinary double points that are fixed by the real involution.
At such curves the isoperiodic sets remain manifolds and are transverse to the curves with double points, so Whitham deformations can change the genus of the spectral curve.
We describe how the double points open, and when the derivative of the normalized Willmore energy vanishes at a nodal curve.
By inserting double points, we show that every value of a discrete invariant of the isoperiodic sets found in previous work occurs.
\end{abstract}

\section{Introduction}

Opening a double point is a natural way to raise the genus of a hyperelliptic curve.
The metric of a constant mean curvature (CMC) torus in $\bbR^3$ is given by a doubly periodic solution of the sinh-Gordon equation, an integrable equation, and such solutions are encoded by spectral curves: real hyperelliptic curves carrying two meromorphic differentials with periods in $2\pi\ci\bbZ$~\cite{Pinkall1989,Hitchin1990,Bobenko1991}.
For these spectral curves, opening double points suggests a way to produce new examples from known ones: insert a double point into a curve of lower genus, and open it while preserving the periods.
This idea has been used to prove the existence of CMC tori of every spectral genus~\cite{Ercolani1993,Jaggy1994,Kewlin2015thesis,Carberry2016a}, and to study moduli spaces of periodic solutions of integrable equations~\cite{Grinevich1995,Grinevich1997arxiv,Kilian2015,Heller2021a}.
We show that the period-preserving deformations of these curves, called Whitham deformations, in the form developed in~\cite{Hasse2026a}, pass smoothly through curves with any number of double points on the unit circle (Theorem~\ref{thm:Whitham deformation node}).
The isoperiodic sets are manifolds at these curves and transverse to them, so the double points open along Whitham deformations and the genus changes.

The best one could hope for is that the isoperiodic sets extend as manifolds transversally through the set of curves with double points.
For nonsingular curves the isoperiodic sets were constructed in~\cite{Hasse2026a}, where we generalized Whitham deformations to all nonsingular real hyperelliptic curves of CMC type.
There one preserves the periods of the plane $\calB_a$ of differentials with purely imaginary periods, which makes sense for every such curve, not only for the spectral curves of tori.
The isoperiodic sets $\calH^g_0(a)$ are then manifolds of dimension $3$, and the isoperiodic sets $\calH^g_0(a,\tau)$ that additionally preserve a complex structure on $\calB_a$ are manifolds of dimension $1$.

We extend this to nodal curves by enlarging $\calH^g_0$ to a set $\calU^g$ that contains the curves with double points on $\bbS^1$, in which those with $n$ double points form a stratum of codimension $n$, and over which $\calB_a$ and its periods are smooth (Section~\ref{sec:extending singular}).
The main difficulty is that at a nodal curve the Whitham equations no longer determine the tangent directions to the isoperiodic set, because they do not see the periods along the cycles through the double points.
We compute the derivatives of these nodal periods (Lemma~\ref{lem:nodal B period derivative}), which add conditions at each node to the Whitham equations (Lemma~\ref{lem:whitham kernel node}).
Under a nondegeneracy condition at the nodes, this best case holds: the isoperiodic sets $\calU^g(a)$ and $\calU^g(a,\tau)$ are submanifolds of dimensions $3$ and $1$, transverse to the strata (Theorem~\ref{thm:Whitham deformation node}).
We also describe how each double point opens, into a conjugate pair of branch points or into two branch points on $\bbS^1$, and find that the possible patterns of opening are rigid.


In the doubly periodic case, the isoperiodic sets of~\cite{Hasse2026a} carry a discrete invariant, a triple of integers $(n_1,n_2,n_3)$, and a continuous invariant $\tau$.
These invariants extend to nodal curves, and we use this to show that every permissible triple $(n_1,n_2,n_3)$ occurs (Theorem~\ref{thm:all triples}).
The proof inserts one node at a time into a doubly periodic curve and opens it along the isoperiodic set, an induction on the genus in the spirit of Jaggy's.

The questions that opening double points has been used to answer are of three kinds.
The first is existence: for example, are there constant mean curvature tori of every spectral genus?
Ercolani, Kn\"orrer and Trubowitz~\cite{Ercolani1993} answered it for even genus, and Jaggy~\cite{Jaggy1994} for every $g\geq2$, by the inverse function theorem near a curve with one double point on $\bbS^1$.
For CMC tori in $\bbS^3$, Carberry and Schmidt~\cite{Carberry2016a} settle existence by an induction on the genus that opens a double point on $\bbS^1$ at each step.
The second kind of question is the shape of the moduli space, which includes finding the minima of natural functionals on it.
For the nonlinear Schr\"odinger equation, Grinevich and Schmidt~\cite{Grinevich1995} describe the moduli space of periodic solutions by starting from the rational curve of the constant potentials, which has infinitely many double points, and opening finitely many of them with period-preserving flows.
They apply the same method to closed curves in $\bbR^3$ through the filament equation~\cite{Grinevich1997arxiv}.
For equivariant constant mean curvature tori in $\bbS^3$, Kilian, Schmidt and Schmitt~\cite{Kilian2015} show that the moduli space is connected, by deforming flat tori of spectral genus $0$ into Delaunay tori of spectral genus $1$ through the opening of a double point.
Heller and Ndiaye~\cite{Heller2021a} construct families of constrained Willmore tori by opening double points on the spectral curve of a homogeneous torus, and show that for conformal classes near that of the square torus they minimize the Willmore energy~\cite{Heller2021b}; see also the survey of Heller and Pedit~\cite{Heller2018a}.
Hauswirth, Kilian and Schmidt~\cite[Section~3.4]{Hauswirth2017arxiv} describe the addition and removal of double points for spectral data with a single differential.
The third kind of question is the closure of the spectral curves of CMC tori among all spectral curves, which lies in $\calS^g$~\cite{Carberry2016}.
For $g=2$ it equals $\calS^2$, whose boundary, which contains curves with double points, is described by Carberry, Kilian, Klein and Schmidt~\cite{Carberry2023arxiv}.
For CMC tori in $\bbS^3$ the closure is the whole space of spectral curves, which is the density result of~\cite{Carberry2016a} mentioned above.
In the setting of constant mean curvature tori, we describe the whole isoperiodic set through a nodal curve, rather than a single deformation that opens it, and we allow any number of double points.
In this paper we will not consider the further conditions at a Sym point that are needed to obtain a CMC torus, as opposed to a doubly periodic solution of the sinh-Gordon equation.

This paper arose from trying to understand the limits of spectral curves of genus $2$ as double points formed on $\bbS^1$, as part of a project continuing from~\cite{Knopf2025}.
Like~\cite{Carberry2019}, it aims to provide general tools for subsequent work.
In particular, it takes up the question, raised in~\cite{Carberry2020}, of connecting the spaces of spectral curves of different genera.
Section~\ref{sec:background} recalls what we need from~\cite{Hasse2026a}.
Section~\ref{sec:extending singular} extends $\calB_a$ and the period map to singular curves, shows by an example that smoothness can fail where two nodes collide, and proves smoothness at curves with nodes on $\bbS^1$.
Section~\ref{sec:isoperiodic nodes} treats the isoperiodic sets, the opening of nodes, and $\willmore$ at nodal curves.
Section~\ref{sec:invariants nodes} treats the invariants at nodal curves.

During the preparation of this work, Claude (Anthropic) was used in the scaffolding, checking, and editing of this article, as well as the creation of figures.

\subsection{Background}
\label{sec:background}

We recall the setup and results of~\cite{Hasse2026a} that we use, in the same notation, and refer there for proofs.
For $m\in\bbN_0$ let $\calP^m_\bbR$ be the space of polynomials $p$ of degree at most $m$ that satisfy the reality condition
\begin{equation}
\label{eq:reality condition}
\lambda^m\,\overline{p(\bar{\lambda}^{-1})} = p(\lambda).
\end{equation}
For $a\in\calP^{2g}_\bbR$ the function $\lambda^{-g}a(\lambda)$ is real on $\bbS^1$.
A polynomial $a\in\calP^{2g}_\bbR$ without a root at $\lambda=0$ defines a hyperelliptic curve $\Sigma_a$ by $\nu^2=\lambda a(\lambda)$, compactified over $\lambda=0,\infty$.
It carries the hyperelliptic involution $\sigma(\lambda,\nu)=(\lambda,-\nu)$ and the real involution $\rho(\lambda,\nu)=(\bar{\lambda}^{-1},\bar{\lambda}^{-(g+1)}\bar{\nu})$, which fixes the points over $\bbS^1$.
The space $\calH^g$ is the set of $a\in\bbC^{2g}[\lambda]$ that satisfy the following conditions (Definition~\refinv{def:Hg}{1.1}):
\begin{itemize}
\item[(H1)]
\emph{Reality condition:} $a\in\calP^{2g}_\bbR$.
\item[(H2)]
\emph{Positivity condition:} $\lambda^{-g}a(\lambda)\geq0$ for all $\lambda\in\bbS^1$.
\item[(H3)]
\emph{Monic condition:} $a(0)=1$.
\end{itemize}
Condition (H3) is a normalization: scaling by $\bbR^\times$ and the rotation $a(\lambda)\mapsto e^{g\ci\phi}a(e^{-\ci\phi}\lambda)$ give isomorphic curves, and after imposing (H3) only the rotations by $g$-th roots of unity remain.
We write $\calH^g_0$ for the subset of $a\in\calH^g$ with $2g$ distinct roots, that is, with nonsingular curve $\Sigma_a$.
It is open in the $(2g-1)$-dimensional affine space $\{a\in\calP^{2g}_\bbR\mid a(0)=1\}$.
The roots of $a\in\calH^g_0$ come in pairs $\alpha,\bar{\alpha}^{-1}$, which are exchanged by $\rho$.
The cycles $A$ on $\Sigma_a$ with $\rho_\ast A=-A$ are called $A$-cycles; they form a subgroup of $H_1(\Sigma_a,\bbZ)$ of rank $g$.
We only use canonical bases of homology $A_1,\dots,A_g,B_1,\dots,B_g$ whose $A_j$ are $A$-cycles.

For $b\in\calP^{g+1}_\bbR$ we consider the differential $\Theta_b$ on $\Sigma_a$, which has the following symmetries under $\sigma$ and $\rho$:
\begin{equation}
\label{def:theta}
\Theta_b = \frac{b(\lambda)}{\lambda}\,\frac{\dd\lambda}{\nu},
\qquad
\sigma^\ast\Theta_b=-\Theta_b,
\qquad
\rho^\ast\Theta_b=-\overline{\Theta_b}.
\end{equation}
Its only poles are residue-free double poles over $\lambda=0,\infty$.
The space $\calB_a$ consists of those $b\in\calP^{g+1}_\bbR$ for which $\Theta_b$ has purely imaginary periods (Definition~\refinv{def:calBa}{1.2}).
As the $A$-periods of $\Theta_b$ are real, and the real parts of its $B$-periods are half-integer linear combinations of its $A$-periods, this is equivalent to the vanishing of the $A$-periods.
By Lemma~\refinv{lem:differential omega}{1.3}, the map $\calB_a\to\bbC$, $b\mapsto b(0)$, is an isomorphism of real vector spaces.
For $\omega\in\bbC$ we write $b^\omega$ for the element of $\calB_a$ with $b^\omega(0)=\omega$, and $J$ for the complex structure $Jb^\omega=b^{\ci\omega}$ on $\calB_a$.
A root common to two linearly independent elements of $\calB_a$ is common to all of $\calB_a$, and we call it a common root of $\calB_a$.
The subset of $\calH^g_0$ where $\calB_a$ has a common root on $\bbS^1$ is denoted $\calS^g$~\cite{Carberry2016}.

To compare homology on nearby curves, fix $a\in\calH^g_0$ and pairwise disjoint balls $B(\alpha,\varepsilon)$ around its roots.
The polynomials in $\calH^g_0$ with exactly one root in each ball form a \emph{distinct root neighborhood} $U$ of $a$.
Over the complement of the balls, the curves $\Sigma_{\tilde{a}}$, $\tilde{a}\in U$, are identified by projection to $\lambda$, so a basis of homology on $\Sigma_a$ is transported to all of $U$.
The period map $\Per_a(b)=(\int_{A_1}\Theta_b,\dots,\int_{B_g}\Theta_b)$ has the two halves (Definition~\refinv{def:period map}{1.4})
\begin{equation}
\label{eq:APer BPer}
\APer_a(b)=\left(\int_{A_1}\Theta_b,\dots,\int_{A_g}\Theta_b\right)\in\bbR^g,
\qquad
\BPer_a(b)=-\ci\left(\int_{B_1}\Theta_b,\dots,\int_{B_g}\Theta_b\right)\in\bbC^g.
\end{equation}
Thus $\calB_a=\ker\APer_a$, and $\BPer_a$ takes values in $\bbR^g$ on $\calB_a$.
Let $\calB \subset \calH^g_0\times\calP^{g+1}_\bbR$ be the kernel of $\APer$, that is, the union of the planes $\{a\}\times\calB_a$.
Lemma~\refinv{lem:differential period}{1.5} states that $\Per$ is smooth in $a$ and that $\calB$ is a trivial rank $2$ vector bundle over $\calH^g_0$.
It also states that $\BPer_a|_{\calB_a}\colon\calB_a\to\bbR^g$ has full rank.
Therefore, for $g\geq2$, the image $\BPer_a[\calB_a]$ is a plane, that is, a point of $\Gr(2,\bbR^g)$, and carries the complex structure transported from $J$.
A real plane $V\subset\bbR^g$ with complex structure $J$ corresponds to the complex line $V^\bbC=\{v-\ci Jv\mid v\in V\}$, so we obtain a point $\BPer_a[\calB_a]^\bbC \in \CP^{g-1}\setminus\RP^{g-1}$.

The isoperiodic sets $\calH^g_0(a)$ and $\calH^g_0(a,\tau)$ are the connected components containing $a$ of the level sets of $\tilde{a}\mapsto\BPer_{\tilde{a}}[\calB_{\tilde{a}}]$ and $\tilde{a}\mapsto\BPer_{\tilde{a}}[\calB_{\tilde{a}}]^\bbC$ on $\calH^g_0$ (Definition~\refinv{def:isoperiodic sets}{2.1}).
They are studied through a basis $(b_1,b_2)$ of $\calB_a$.
The periods of every $b\in\calB_a$ are real linear combinations of those of $b_1$ and $b_2$, so fixing the periods of $\Theta_{b_1}$ and $\Theta_{b_2}$ fixes the plane $\BPer_a[\calB_a]$.
Conversely, for $g\geq2$, if $\BPer_{\tilde{a}}[\calB_{\tilde{a}}]=\BPer_a[\calB_a]$ then, as $\BPer_{\tilde{a}}|_{\calB_{\tilde{a}}}$ is injective, $\calB_{\tilde{a}}$ has a basis whose $B$-periods agree with those of $b_1$ and $b_2$, and their $A$-periods all vanish.
Hence $\calH^g_0(a)$ is the projection to $\calH^g_0$ of the level set through $(a,b_1,b_2)$ of the map
\begin{align}
\label{eq:period map}
\Phi:\calH^g_0\times(\calP^{g+1}_{\bbR})^2&\to\Hom(H_1(\Sigma_a,\bbZ),\bbC^2),&
(a,b_1,b_2)&\mapsto\left(C\mapsto\int_C \Theta_{b_1},\int_C \Theta_{b_2}\right).
\end{align}
Lemma~\refinv{lem:whitham kernel}{2.2} describes the kernel of $\dd\Phi$ in three parts, a format that Lemma~\ref{lem:whitham kernel node} below mirrors for nodal curves.
Part~(i) says that $(\dot{a},\dot{b}_1,\dot{b}_2)$ lies in the kernel if and only if $\dot{\Theta}_{b_k}=\dd(\ci c_k/\nu)$ for some $c_k\in\calP^{g+1}_\bbR$, which is equivalent to the \emph{Whitham equations}
\begin{equation}
\label{eq:whitham ab}
2a\dot{b}_k - \dot{a}b_k = \ci(2\lambda a c_k' - (\lambda a' + a)c_k),
\qquad k=1,2.
\end{equation}
Antisymmetrizing in $k$ gives the \emph{compatibility equation}
\begin{equation}
\label{eq:whitham Q}
b_2c_1 - b_1c_2 = aQ,
\end{equation}
for some $Q\in\calP^2_\bbR$.
The polynomials $c_1$, $c_2$ and $Q$ are uniquely determined by $(\dot{a},\dot{b}_1,\dot{b}_2)$.
Part~(ii) says that if $g\geq1$ and $b_1$ and $b_2$ have no common root, then $(\dot{a},\dot{b}_1,\dot{b}_2)\mapsto Q$ is an isomorphism of the kernel onto $\calP^2_\bbR$.
Part~(iii) says that if $g\geq1$ and $b_1$ and $b_2$ have exactly one common root $\beta$, of first order, then the kernel is still $3$-dimensional, but $Q$ maps it onto $\{Q\in\calP^2_\bbR\mid Q(\beta)=0\}$.

Using this description of the kernel to control the rank of $\Phi$, Theorem~\refinv{thm:Whitham deformation}{2.3} shows that for $g\geq2$, if $\calB_a$ has at most one common root, of first order, then $\calH^g_0(a)$ is locally a three-dimensional submanifold of $\calH^g_0$.
Along $\calH^g_0(a)$, take a frame $(b_1,b_2)$ of $\calB$ with constant periods and let $\tau=b_2(0)/b_1(0)$.
Then $\calH^g_0(a,\tau)$ is the subset of $\calH^g_0(a)$ on which $\tau$ is constant, and comparing constant coefficients in~\eqref{eq:whitham ab} and~\eqref{eq:whitham Q} gives
\begin{equation}
\label{eq:dot tau Q}
\dot{\tau}
= \frac{\ci a(0)}{2b_1(0)^2}Q(0).
\end{equation}
Under the same hypotheses, $\calH^g_0(a,\tau)$ is locally a one-dimensional submanifold of $\calH^g_0$.
By~\eqref{eq:dot tau Q}, its tangent line consists of the vectors with $Q\in\bbR\lambda$.

If the $b\in\calB_a$ for which $\Theta_b$ has periods in $2\pi\ci\bbZ$ form a lattice of rank $2$, then $\Sigma_a$ is the spectral curve of a doubly periodic solution of the sinh-Gordon equation, that is, of a CMC immersion of the plane with doubly periodic metric.
For a basis $(b_1,b_2)$ of this lattice, the Willmore energy of a fundamental domain is $2\ci\Res_{\lambda=0}q_{b_2}\,\Theta_{b_1}$, where $q_{b_2}$ is a local primitive of $\Theta_{b_2}$~\cite[Theorem~5.7]{Knopf2025}.
For any $b_1,b_2\in\calB_a$ this residue is real, by the reciprocity law and the real involution.
Dividing by the signed area of the parallelogram in $\bbR^g$ spanned by $\BPer_a(b_1)$ and $\BPer_a(b_2)$ gives the normalized Willmore energy (Definition~\refinv{def:willmore}{2.4})
\begin{equation}
\label{eq:willmore}
\willmore(a)
= \frac{2\ci\Res_{\lambda=0}q_{b_2}\,\Theta_{b_1}}{\operatorname{Area}(\BPer_a(b_1),\BPer_a(b_2))}
= \frac{8\ci\bigl(b_1(0)b_2'(0) - b_1'(0)b_2(0)\bigr)}{\operatorname{Area}(\BPer_a(b_1),\BPer_a(b_2))},
\end{equation}
which does not depend on the basis $(b_1,b_2)$ of $\calB_a$, since a change of basis multiplies the numerator and the denominator by its determinant.
Along $\calH^g_0(a)$, with a frame $(b_1,b_2)$ of constant periods, the area is constant and $\dot{\willmore}\operatorname{Area} = 4Q'(0)$.
Theorem~\refinv{thm:Willmore derivative}{2.5} states that the derivative $\dot{\willmore}$ along $\calH^g_0(a,\tau)$ vanishes at $a$ if and only if $\calB_a$ has a common root $\beta \in \bbS^1$, and in that case $\ddot{\willmore}$ vanishes if and only if $\frac{\dd}{\dd\lambda}\frac{b_2}{b_1}$ vanishes at $\beta$.
For an oriented basis $(b_1,b_2)$ of $\calB_a$, the function $f = b_2/b_1$ is real on $\bbS^1$ and has $f(0)=\tau$.
It is the winding number function of~\cite[Section~3]{Carberry2016}, up to taking the reciprocal.

\section{Extending the Period Map to Singular Curves}
\label{sec:extending singular}

In this section we extend the definition of $b^\omega$, $\calB_a$, and the period map $\Per_a$ to parts of $\calH^g \setminus \calH^g_0$.
Lemma~\ref{lem:b_omega continuous} shows that $\omega \mapsto b^\omega$ extends continuously to all of $\calH^g$, and that $b^\omega$ vanishes at the singularities of $\Sigma_a$.
Example~\ref{ex:colliding nodes} shows, however, that the period map need not be smooth where two nodes on $\bbS^1$ come together.
We therefore restrict subsequent statements to curves with nodes (ordinary double points) on $\bbS^1$.
The main technical device is to enlarge $\calH^g_0$ to a set $\calU^g$ that also contains these nodal curves, in which the curves with $n$ nodes on $\bbS^1$ form a stratum of codimension $n$.
After setting up homology at the nodes, with $A$- and $B$-cycles that vary smoothly as the nodes open, Lemma~\ref{lem:b_omega smooth node} shows that $\calB$ is a smooth vector bundle over $\calU^g$ and that $\BPer$ extends smoothly, with full rank as in Lemma~\refinv{lem:differential period}{1.5}.
We close by noting that the results of~\cite[Section~2]{Hasse2026a} on isoperiodic sets hold on $\calU^g_0$.
We begin with the continuity result, which holds for all singular curves.
The argument is a generalization of~\cite[Lemma~3.4]{Knopf2025}.

\begin{lemma}
\label{lem:b_omega continuous}
The function $(a,\omega) \mapsto b^\omega$ extends to a continuous function on $a \in \calP^{2g}_\bbR$ with $a(0)=1$ and $\omega \in \bbC$.
Further, $b^{\omega}$ has a root of order at least $\lfloor \frac{m}{2} \rfloor$ at a root of $a$ of order $m$. 
\end{lemma}
\begin{proof}
For all $a \in \calP^{2g}_\bbR$ with $a(0)=1$ and $2g$ distinct roots and for all $\omega \in \bbC$ there exists a harmonic function $h^\omega$ on $\Sigma_a^\circ$ with $\dd h^\omega = \Real \Theta_{b^\omega}$, because the periods of $\Theta_{b^\omega}$ are purely imaginary.
The additive constant of $h^\omega$ is further fixed by the condition $\sigma^\ast h^\omega = - h^\omega$.
Any $a \in \calP^{2g}_\bbR$ with $a(0) = 1$ is the limit of a sequence $(a_n)_{n \in \bbN}$ where $a_n$ additionally has $2g$ distinct roots. 
Write $b^\omega_n$ and $h^\omega_n$ for $b^\omega$ and $h^\omega$ at $a_n$.
Choose a norm on $\calP^{g+1}_\bbR$; then we may pass to a convergent subsequence of $b^\omega_n/\|b^\omega_n\|$ with nontrivial limit $\tilde{b}^\omega$.
Away from the roots of $a$ and $\lambda = 0,\infty$, every point $\lambda$ has a compact neighborhood on which the harmonic functions $h^\omega_n/\|b^\omega_n\|$ converge uniformly to $\tilde{h}^\omega$.
Moreover, away from the roots of $a$ and $\lambda = 0, \infty$ we know that $\dd\tilde{h}^\omega = \Real \Theta_{\tilde{b}^\omega}$.
Next, because $a(0)=1$, there are small discs around $\lambda = 0,\infty$ that contain no roots of the $a_n$.
By the strong maximum principle, $|h^\omega_n|$ is bounded on the complement of these discs by its values on the boundaries of the discs.
Therefore $\tilde{h}^\omega$ is bounded outside these discs, and in particular we may apply the removable singularity theorem at the roots of $a$ to conclude that this limit is a harmonic function on the normalization of $\Sigma_a^\circ$.

We next show by contradiction that $\|b^\omega_n\|$ is bounded.
Suppose that $\|b^\omega_n\| \to \infty$.
In a small disc around $\lambda=0$ (a similar argument holds around $\lambda = \infty$) we have the Laurent series expansion $h^\omega_n \sim - \Real (2\nu^{-1}(\omega + \dots))$.
Hence $h^\omega_n / \|b^\omega_n\|$ converges to a harmonic function on this small disc, as $\omega/\|b^\omega_n\| \to 0$ by assumption.
Therefore $h^\omega_n/\|b^\omega_n\|$ converges to a harmonic function on all of the normalization of $\Sigma_a$. 
Hence it is constant, and by the hyperelliptic symmetry $\sigma$ it is zero.
Then the meromorphic differential $\Theta_{\tilde{b}^\omega}$ is also zero, and hence $\tilde{b}^\omega = 0$, a contradiction.
Therefore $\|b^\omega_n\|$ is bounded, and we may pass to a convergent subsequence of $b^\omega_n$ itself.

Suppose that two subsequences of $(b^\omega_n)$ had different limits.
We would have two harmonic functions on the normalization of $\Sigma_a^\circ$.
By the above Laurent expansion, their difference is also harmonic at $\lambda = 0,\infty$, since the singular part there is determined by $\omega$ alone.
It follows that this difference is everywhere zero.
In other words, the sequence $(b^\omega_n)$ has a unique accumulation point.
A bounded sequence with a unique accumulation point is convergent.
Thus we have constructed the unique limit of $b^\omega_n$ and thereby extended $\calB_a$ to all $a \in \calP^{2g}_\bbR$ with $a(0)=1$.
Continuity at $a$ with distinct roots is Lemma~\refinv{lem:differential period}{1.5}; at general $a$ it follows from the above by a diagonal argument, since polynomials with distinct roots are dense.

In addition, this argument shows that $\Theta_{b^\omega}$ is a meromorphic differential on the normalization of $\Sigma_a$ with poles only over $\lambda = 0,\infty$.
This forces $b^\omega$ to have roots at the roots of $a$ to the stated orders.
\end{proof}

Next we consider whether $b^\omega$ and the period map are smooth, and hence whether the planes $\calB_a$ form a smooth vector bundle over $\calH^g$.
This turns out not to be possible in the same generality as the above lemma, as the following example shows.

\begin{example}
\label{ex:colliding nodes}
Consider $a(\lambda) = (\lambda^2 - 2c\lambda + 1)^2$ for $c>0$.
Since $\lambda^{-2}a(\lambda) = (\lambda + \lambda^{-1} - 2c)^2$, this is an element of $\calH^2$.
For $c < 1$, we have two nodes at $e^{\pm\ci\varphi}$, where $c = \cos\varphi$ and $0 < \varphi < \frac{\pi}{2}$.
In the limit $c \to 1$ they come together, giving $a(\lambda) = (\lambda-1)^4$.
The normalization of $\Sigma_a$ has genus $0$, and we parametrize it by $\zeta \in \CP^1$ using $\lambda = \zeta^2$ and $\nu = \zeta(\zeta^4 - 2c\zeta^2 + 1)$.
By Lemma~\ref{lem:b_omega continuous}, $b^\omega$ is divisible by $\lambda^2 - 2c\lambda + 1$, and so
\[
b^\omega(\lambda) = (\lambda^2 - 2c\lambda + 1)(\omega + \bar{\omega}\lambda),
\qquad
\Theta_{b^\omega} 
= 2\left(\omega\zeta^{-2} + \bar{\omega}\right)\dd\zeta
= \dd \left( 2(\bar{\omega}\zeta - \omega\zeta^{-1}) \right).
\]
In this example $b^\omega$ is a smooth function of $c$, including at the limit $c = 1$, where it has a root of order $2 = \lfloor \frac{4}{2} \rfloor$ in accordance with Lemma~\ref{lem:b_omega continuous}.
The $B$-cycle at the node $\lambda = e^{\ci\varphi}$ is a path from $\zeta = -e^{\ci\varphi/2}$ to $\zeta = e^{\ci\varphi/2}$, and similarly for $\lambda = e^{-\ci\varphi}$.
Thus, up to the signs fixed by orientations,
\[
\BPer_a(b^\omega)
= -8\left( \Imag\omega \cos\tfrac{\varphi}{2} - \Real\omega \sin\tfrac{\varphi}{2}, \Imag\omega \cos\tfrac{\varphi}{2} + \Real\omega \sin\tfrac{\varphi}{2} \right).
\]
In terms of the coefficients of $a$ we have $\cos\frac{\varphi}{2} = \sqrt{(1+c)/2}$, which is smooth at $c=1$, but $\sin\frac{\varphi}{2} = \sqrt{(1-c)/2}$, which is not.
Hence $\BPer_a(b^\omega)$ is continuous but not smooth at $c = 1$.
Moreover, the two components agree in the limit, so the limit of $\BPer_a|_{\calB_a}$ does not have full rank.
This is also observed in the proof of~\cite[Theorem~3.5]{Knopf2025}.
The example demonstrates that any generalization of Lemma~\refinv{lem:differential period}{1.5} to singular curves requires additional conditions.
\end{example}

Other singular curves can behave worse still: for equivariant non-conformal harmonic tori in $\bbS^3$, the space of spectral curves spirals around a curve at which branch points come together off $\bbS^1$~\cite{Carberry2020}.
Cautioned by these examples, we restrict ourselves to curves whose only singularities are nodes on $\bbS^1$.
In order to employ the inverse function theorem, it is useful to enlarge this set slightly to an open set within $\{ a \in \calP^{2g}_\bbR \mid a(0) = 1\}$, by replacing condition (H2) with the following condition (H2').

\begin{definition}
\label{def:U}
Let $\calU^g$ be the set of $a \in \bbC^{2g}[\lambda]$ that satisfy conditions (H1), (H3), and
\begin{itemize}
\item[(H2')] 
$\lambda^{-g}a(\lambda) > 0$ for some $\lambda \in \bbS^1$.
Double roots of $a$ are allowed on $\bbS^1$, but all other roots are simple. 
Any unimodular double root of $a$ is a local minimum of $\lambda^{-g}a(\lambda)$. 
\end{itemize}
For $n\geq0$, let $\calH^g_n$ be the subset of $\calH^g$ where $a$ has exactly $n$ distinct double roots on $\bbS^1$ and otherwise simple roots.
Likewise, let $\calU^g_n$ be the subset of $\calU^g$ with exactly $n$ double roots on $\bbS^1$.
Each $\calH^g_n$ is contained in $\calU^g_n$, which in addition contains polynomials with simple roots on $\bbS^1$.
\end{definition}

In particular, every $a\in\calU^g_0$ has only simple roots, some of which may lie on $\bbS^1$.
We can extend all the results of Section~\refinv{sec:differentials}{1.2} to $\calU^g_0$ with minimal modifications.
The real-valued function $\lambda^{-g}a(\lambda)$ on the circle must change sign an even number of times, so the number of simple roots on $\bbS^1$ is even.
Consider a maximal arc on the unit circle where $\lambda^{-g}a(\lambda)$ is negative.
Its endpoints are simple roots of $a$; we say these roots are a pair, akin to $\alpha,\bar{\alpha}^{-1}$.
If $\tilde{\gamma}$ is a lift of this arc, then the real involution $\rho$ acts on it as $\sigma$.
Hence $A = \tilde{\gamma} - \sigma_\ast\tilde{\gamma}$ obeys $\rho_\ast A = - A$, so it is an $A$-cycle, and the corresponding period of a differential of the form~\eqref{def:theta} is real.
Thus we still have a notion of a canonical basis with $A$- and $B$-cycles, and the arguments of Lemmas~\numinv{lem:differential omega}{1.3} and~\refinv{lem:differential period}{1.5} apply.

Recall that the $A$-cycles of an adapted canonical basis (Definition~\refinv{def:adapted basis}{3.1}) are lifts $\tilde{\gamma}-\sigma_\ast\tilde{\gamma}$ of paths $\gamma$ in the $\lambda$-plane that are invariant under $\lambda\mapsto\bar{\lambda}^{-1}$ and join a pair of roots.
The cycles of the previous paragraph are of exactly this form.
The proof of Lemma~\refinv{lem:CA adapted basis}{3.2}, that two adapted canonical bases differ by an element of $\CA(g)$, goes through with one modification.
If there are roots of $a$ on $\bbS^1$, then there are multiple possible paths over the unit circle in the construction of $\tilde{\gamma}'$.
It does not matter which path is used, only that the same path is retraced on the return journey.

We are now ready to handle nodal curves.
By a node we always mean an ordinary double point.
A double root $\alpha\in\bbS^1$ of $a$ corresponds to a node of $\Sigma_a$ over $\lambda=\alpha$, and we denote both by $\alpha$.
Suppose that two branch points come together at a node on $\bbS^1$.
The cycle around them shrinks to the node, and we call it the \emph{neck cycle}.
It is an $A$-cycle, and hence this $A$-period of $\Theta_b$ is zero for $b \in \calB_a$.
For general $b \in \calP^{g+1}_\bbR$, $\Theta_b$ pulled back to the normalization may have a simple pole over the node, with opposite residues at the two points over it.
The neck cycle surrounds this pole, and thus the $A$-period is $2\pi\ci$ times the residue.
This gives another proof of the fact that $b^\omega$ has a root at any double root of $a$ on $\bbS^1$, a special case of Lemma~\ref{lem:b_omega continuous}.
We define the \emph{$B$-cycle} to be the path in the normalization of $\Sigma_a$ that begins and ends at the two points over the node, and the \emph{$B$-period} to be the integral along this path.
It is well known that these are the appropriate definitions for homology on nodal curves, but we will see this demonstrated in Lemma~\ref{lem:b_omega smooth node}.

We next give a family of $A$- and $B$-cycles that are well-behaved in the presence of nodes.
Let $a \in \calU^g$ have double roots $\alpha_1,\dots,\alpha_n \in \bbS^1$.
Fix a ball $B(\alpha_j,3\varepsilon)$ around each double root $\alpha_j$ of $a$ and a ball of radius $\varepsilon$ around each simple root of $a$, with $\varepsilon$ small enough that all these balls are pairwise disjoint.
Let $U\subset\calU^g$ be a neighborhood of $a$ such that each ball $B(\alpha_j,\varepsilon)$ contains exactly two roots $\alpha$ and $\alpha^*$ of $\tilde{a}\in U$, counted with multiplicity, and each ball around a simple root of $a$ contains exactly one root of $\tilde{a}$.
Since $\tilde{a}$ has $2g$ roots, it has no roots in $B(\alpha_j,3\varepsilon)\setminus B(\alpha_j,\varepsilon)$.
A lift of $\partial B(\alpha_j,2\varepsilon)$ to $\Sigma_{\tilde{a}}$ is an $A$-cycle, and for $\tilde{a} = a$ it is the neck cycle of the node.
For $\alpha\neq\alpha^*$ this cycle is homologous to $\tilde{\gamma}-\sigma_\ast\tilde{\gamma}$, where $\gamma$ is either the radial segment or the arc of $\bbS^1$ from $\alpha$ to $\alpha^*$.
Thus it can be an element of an adapted canonical basis.
Fix a direction that is neither tangent nor radial to $\bbS^1$ at any point of $\bbS^1 \cap B(\alpha_j,\varepsilon)$.
Consider the chord of $B(\alpha_j,3\varepsilon)$ through $\frac{1}{2}(\alpha + \alpha^*)$ in this direction (see Figure~\ref{fig:chord}).
Because the chord is defined using only symmetric functions of the two roots, it is a smooth function of the coefficients of $\tilde{a}$, including when the two roots come together.
Moreover, this chord either passes through the node, when there is one, or separates $\alpha$ and $\alpha^*$ when they are distinct.
The lift of this chord to $\Sigma_{\tilde{a}}$ has intersection number $\pm 1$ with this $A$-cycle, and we choose the orientation to make it $1$.
It can be completed to a cycle on $\Sigma_{\tilde{a}}$ that avoids the $\varepsilon$-balls around all other roots of $a$, and so varies smoothly with $\tilde{a}$.
We therefore choose a canonical basis of homology that varies smoothly on $U$ and has $A$- and $B$-cycles of the above form at the double roots of $a$.
This extends the notion of distinct root neighborhoods to all of $\calU^g$.

\begin{figure}
\centering
\newcommand{\chordpanel}[2]{%
  \coordinate (C) at (-3,0);
  \coordinate (r1) at ($(C)+(#1)$);
  \coordinate (r2) at ($(C)+(#2)$);
  \coordinate (m) at ($(r1)!0.5!(r2)$);
  \begin{scope}
    \clip (0,0) circle (1.3);
    \fill[gray!12] (C) circle (3);
    \draw[thick, gray] (C) circle (3);
  \end{scope}
  \begin{scope}
    \clip (0,0) circle (1.25);
    \draw[thick, dashed] (m) -- ($(m)+(55:2)$);
    \draw[thick] ($(m)+(55:-2)$) -- (m);
  \end{scope}
  \draw (0,0) circle (1);
  \node[gray, above left] at ($(C)+(24:3)$) {$\bbS^1$};
}
\newcommand{\rootsegment}{%
  \draw[gray!50] (r1) -- (r2);
}
\newcommand{\rootcross}[1]{%
  \draw[thick] ($(#1)+(-0.07,-0.07)$) -- ($(#1)+(0.07,0.07)$);
  \draw[thick] ($(#1)+(-0.07, 0.07)$) -- ($(#1)+(0.07,-0.07)$);
}
\begin{subfigure}{0.32\textwidth}
\centering
\begin{tikzpicture}[scale=1.5]
  \chordpanel{-6:2.55}{-8:3.53}
  \rootsegment \rootcross{r1} \rootcross{r2}
  \node[above left]  at (r1) {$\alpha$};
  \node[above right] at (r2) {$\alpha^*$};
\end{tikzpicture}
\caption{Conjugate pair}
\end{subfigure}
\hfill
\begin{subfigure}{0.32\textwidth}
\centering
\begin{tikzpicture}[scale=1.5]
  \chordpanel{0:3}{0:3}
  \rootcross{r1}
  \node[below right] at (r1) {$\alpha_j$};
\end{tikzpicture}
\caption{Double root on $\bbS^1$}
\end{subfigure}
\hfill
\begin{subfigure}{0.32\textwidth}
\centering
\begin{tikzpicture}[scale=1.5]
  \chordpanel{15:3}{-8:3}
  \rootsegment \rootcross{r1} \rootcross{r2}
  \node[left=3pt] at (r1) {$\alpha$};
  \node[right=3pt] at (r2) {$\alpha^*$};
\end{tikzpicture}
\caption{Two roots on $\bbS^1$}
\end{subfigure}
\caption{
The chord defining the $B$-cycle through a double root, near the ball $B(\alpha_j,\varepsilon)$ that contains the branch points.
The chord continues to the boundary of $B(\alpha_j,3\varepsilon)$.
The shaded region is the inside of the unit disc.
The branch cut (faint gray line) is the segment between the branch points, so that the chord changes sheets (solid vs. dashed) at the midpoint $\frac{1}{2}(\alpha+\alpha^*)$.
The $B$-cycle has intersection $+1$ with the $A$-cycle, which is the lift of $\partial B(\alpha_j,2\varepsilon)$.
}
\label{fig:chord}
\end{figure}

Several of the arguments below pass to the normalization of a nodal curve, and we now set up notation for this.
Consider a factorization $a = p^2 \tilde{a}$ with $p \in \calP^{n}_\bbR$, where the roots of $\tilde{a}$ are not necessarily simple.
Then the (partial) normalization $\Sigma_{\tilde{a}} \to \Sigma_a$ is the lift of the projection to $\lambda$.
Since $\lambda^{-n}p(\lambda)^2 = |p(\lambda)|^2 \geq 0$ on $\bbS^1$, the polynomial $\tilde{a} \in \calP^{2\hat{g}}_\bbR$ satisfies conditions (H1) and (H2) for $\hat{g} = g - n$, but not necessarily (H3).
Only real scalars preserve the reality condition, so we cannot in general rescale $p$ to achieve $\tilde{a}(0) = 1$.
Instead we restore (H3) by rotation, defining $\hat{a} \in \calH^{\hat{g}}$ by
\begin{equation}
\label{eq:normalization}
\hat{a}(\lambda) = e^{\hat{g}\ci\phi}\,\tilde{a}(e^{-\ci\phi}\lambda) \text{ for } e^{\ci\phi} = \tilde{a}(0)^{-1/\hat{g}},
\end{equation}
for a choice of the root.
Given $a$ and $\hat{a}$, there is a natural isomorphism between $\calB_a$ and $\calB_{\hat{a}}$.
Lemma~\ref{lem:b_omega continuous} shows that the polynomial $b^\omega$ is divisible by $p$.
Removing this factor $b = p \tilde{b}$ and rotating gives
\[
\hat{b}(\lambda)
= e^{(\hat{g}+1)\ci \phi/2} \,\tilde{b}(e^{-\ci \phi} \lambda),
\]
so that $\Theta_{\hat{b}}$ is the pullback of $\Theta_b$ under the normalization map and rotation $\lambda \mapsto e^{-\ci\phi}\lambda$.
Since $p(0)^2 = \tilde{a}(0)^{-1} = e^{\hat{g}\ci\phi}$, we have $\hat{b}(0) = \pm e^{\ci\phi/2}\, b(0)$, where the choice of sign, in essence, is the choice of $\pm \hat{\nu}$.
Hence the isomorphism $\calB_a \to \calB_{\hat{a}}$ sends $\omega$ to $e^{\ci\phi/2}\omega$, and in particular is complex linear with respect to $J$.
In the most common case, where $p = \ci \alpha^{-1/2}(\lambda - \alpha)$ for a single node $\alpha \in \bbS^1$, we have $p(0)^2 = -\alpha = \tilde{a}(0)^{-1}$, $e^{\ci\phi} = (-\alpha)^{1/\hat{g}}$, and $\hat{a}(\lambda) = -\alpha \tilde{a}(e^{-\ci\phi}\lambda)$.

\begin{remark}
\label{rem:normalized differentials}
The proof of Lemma~\ref{lem:b_omega smooth node} below uses the normalized holomorphic differentials $\Omega_j$ on a family of algebraic curves.
A standard reference from the complex analytic perspective is~\cite{Fay1973}, and from the algebraic geometric perspective one has~\cite{Serre1988}.
We collect some facts from~\cite{Hu2020} that we will use; we rely on this source because it explicitly handles multiple nodes, whereas~\cite[Proposition~3.7]{Fay1973} is stated only for one node and the subsequent example on page~54 with multiple nodes has rational normalization.
These holomorphic differentials are normalized by $\int_{A_k}\Omega_j = \delta_{jk}$ for $j, k = 1,\dots,g$.
The holomorphic differentials are uniquely determined by their $A$-periods, and vary analytically in the family~\cite[Theorem~1.1 and Remark~2.4]{Hu2020}.
This gives a normalized basis of holomorphic differentials on the nearby curves~\cite[Section~4.2]{Hu2020}.
Suppose that $a$ has double roots at $\alpha_1,\dots,\alpha_n \in \bbS^1$ and write the differential $\Omega_j = u_j\,\frac{\dd\lambda}{\nu}$ with $u_j\in\bbC^{g-1}[\lambda]$.
On the normalization of $\Sigma_a$, for $j\leq n$ the differential $\Omega_j$ is of the third kind, with simple poles only at the two points over $\alpha_j$ and residues $\pm\frac{1}{2\pi\ci}$ there.
For $j>n$ it is holomorphic on the normalization.
The pullback of the $A_j$-period gives the residue, hence $u_j(\alpha_k) = 0$ for $k\leq n$ with $k\neq j$, and $u_j(\alpha_j)\neq0$ for $j\leq n$.
\end{remark}

\begin{lemma}
\label{lem:b_omega smooth node}
At all $a \in \calU^g$ the function $\omega \mapsto b^{\omega}$ is smooth and $\calB = \calU^g \times \calB_a$ (the bundle with fibers $\calB_a$ over $\calU^g$) is a smooth rank $2$ vector bundle.
The period map $\BPer$ extends to a smooth function on $\calB$.
Further, $\BPer_a|_{\calB_a} : \calB_a \to \bbR^g$ has full rank and in particular is injective for $g \geq 2$.
\end{lemma}
\begin{proof}
When $a$ has only simple roots, this is Lemma~\refinv{lem:differential period}{1.5} or its extension to $\calU^g_0$.
So consider $a$ with nodes $\alpha_1,\dots,\alpha_n \in \bbS^1$.
Let $U \ni a$ be a distinct root neighborhood, with its smoothly varying canonical basis of homology, constructed above.
The first statement of the lemma follows from the claim that the following map is smooth and a local diffeomorphism at $(a,b)$
\begin{align*}
\calU^g \times \calP^{g+1}_{\bbR} &\to \calU^g \times \bbC \times \bbR^g, &
(a,b) &\mapsto \left(a, b(0), \left(\int_{A_1} \Theta_b, \dots, \int_{A_g} \Theta_b \right)\right).
\end{align*}
It is smooth on $U$ because the $A$-cycles do not pass through the roots of any $\tilde{a}\in U$.
Since both spaces are submanifolds of dimension $3g+1$, the map is a local diffeomorphism if the kernel of the derivative is trivial at $(a,b)$. 
Let $(\dot{a}, \dot{b})$ be an element of this kernel. 
From the first two components of the map, $\dot{a} = 0$ and $\dot{b}(0) = 0$. 
In a kernel direction, the $A$-periods are constant, hence
\begin{align*}
0
= \frac{\dd}{\dd t} \int_{A_j} \Theta_b
= \int_{A_j} \dot{\Theta}_{b} 
= \int_{A_j} \frac{1}{\lambda}\left( \dot{b} - \frac{b\, \dot{a}}{2a} \right) \frac{\dd\lambda}{\nu}
= \int_{A_j} \frac{\dot{b}}{\lambda}\,\frac{\dd\lambda}{\nu}
= 2\pi\ci\, \dot{b}(\alpha_j) \Res_{\lambda = \alpha_j} \left(\frac{\dd\lambda}{\lambda\nu} \right).
\end{align*}
This forces $\dot{b}$ to vanish at each double root $\alpha_j$ of $a$.
Since $\dot{b}(0) = 0$ too, $\dot{\Theta}_b$ is a holomorphic 1-form on the normalization of $\Sigma_a$. 
It vanishes, since its periods over the remaining $A$-cycles, which are the $A$-cycles of the normalization, are zero. 
This completes the proof of the first statement.

For the second statement, we must further show that $\BPer$ is a smooth function.
We use the $B$-cycles through the chords of Figure~\ref{fig:chord}.
It follows from Lemma~\ref{lem:b_omega continuous} that $\Theta_b$ does not have poles on the normalization over the nodes on $\bbS^1$.
In particular, the integral of $\Theta_b$ along each $B$-cycle is well defined even when $\tilde{a}$ has double roots.
To see that it is smooth, we apply the reciprocity law~\cite[p.~241]{Griffiths1994} to $\Theta_b$ and one of the normalized differentials $\Omega_j = u_j\,\frac{\dd\lambda}{\nu}$ from Remark~\ref{rem:normalized differentials}, with $u_j\in\bbC^{g-1}[\lambda]$.
The reciprocity law works by cutting the surface into a simply connected polygon on which $q_b = \int \Theta_b$ is well-defined.
Then
\begin{align*}
\labelthis{eq:B period residue}
- \int_{B_j}\Theta_b
= \sum_i \left( \int_{A_i}\Theta_b \int_{B_i} \Omega_j - \int_{B_i}\Theta_b \int_{A_i} \Omega_j \right)
&= 2\pi\ci \Res \left( \Omega_j q_b \right) \\
&= 2\pi\ci \left(-2 u_j(0) \omega + 2 u_j(\infty) \bar{\omega}\right),
\end{align*}
where $\omega = b(0)$ and $u_j(\infty)$ means the coefficient of $\lambda^{g-1}$ (the highest order generically).
This formula holds for nonsingular curves and on singular curves for the cycles $B_j$ that do not pass through a node.
To handle these remaining `periods', we pull back to the normalization of $\Sigma_a$, where $B_j$ is a path between the two points $\alpha^\pm$ over the node $\alpha$ and the corresponding differential $\Omega_j$ is a normalized differential of the third kind.
Reciprocity gives in this case
\begin{align*}
0 = 2\pi\ci \Res \left( \Omega_j q_b \right)
&= 2\pi\ci \left(-2\omega u_j(0) + 2 \bar{\omega} u_j(\infty)\right)
+ q_b(\alpha^+) - q_b(\alpha^-)\\
&= 2\pi\ci \left(-2\omega u_j(0) + 2 \bar{\omega} u_j(\infty)\right)
+ \int_{B_j} \Theta_b.
\end{align*}
This formula agrees with~\eqref{eq:B period residue}; it holds for nonsingular and singular curves alike.
Therefore the $B$-periods are smooth, even at $a$, as the $u_j$ are analytic functions of the coefficients of $a$ (see Remark~\ref{rem:normalized differentials}).
This smoothness argument is essentially the same as~\cite[p.~86]{Ercolani1993}.

Finally we compute the rank of $\BPer_a|_{\calB_a}$ at $a \in \calU^g_n$.
Let $(\hat{a},\hat{b})$ be obtained from $(a,b)$ by removing the factors at all $n$ nodes and rotating, as in~\eqref{eq:normalization}, so that $\Sigma_{\hat{a}}$ is the normalization of $\Sigma_a$, of genus $\hat{g}=g-n$, and $\Theta_{\hat{b}}$ is the pullback of $\Theta_b$.
Suppose $b\in\calB_a$ has $\BPer_a(b)=0$.
Then $\Theta_b$ is exact, as is $\Theta_{\hat{b}}=\dd q$, where $q$ is a meromorphic function on $\Sigma_{\hat{a}}$ with $\sigma^\ast q=-q$ and simple poles over $\lambda=0,\infty$.
The $B$-period of $\Theta_b$ corresponding to the node $\alpha_j$ is that along the path from $\alpha_j^-$ to $\alpha_j^+$, and equals $q(\alpha_j^+)-q(\alpha_j^-)=2q(\alpha_j^+)$.
As in the proof of Lemma~\refinv{lem:differential period}{1.5}, there are three cases for the existence of $q$.
If $\hat{g}\geq2$, which covers $1 \leq n \leq g-2$, there is no such $q$ other than $0$.
Next, if $\hat{g}=1$, that is, $n = g-1$, then $q$ is a real multiple of $\lambda^{-1}\hat{\nu}$.
A single node already forces $q=0$, since $\lambda^{-1}\hat{\nu}\neq0$ at $\alpha_j$.
Lastly, if $\hat{g}=0$, that is, $n = g$, then $q=\ci\hat{c}/\hat{\nu}$ for some $\hat{c}\in\calP^1_\bbR$.
So for $2 \leq n=g$ distinct nodes, evaluation at two of them forces $\hat{c}=0$.
In all the cases thus far, the kernel is trivial.
However, for the remaining case $n=g=1$ there is only the single condition $\hat{c}(\alpha_1)=0$, which leaves a one-dimensional kernel, and the rank is $1=\dim \bbR^g$.
Therefore in all cases $\BPer_a|_{\calB_a}$ has full rank, as claimed.
\end{proof}

The results of Section~\refinv{sec:isoperiodic sets}{2} on isoperiodic sets are stated there for $\calH^g_0$, but they hold on $\calU^g_0$ with the same proofs.
Those proofs use only the following properties of $a$: its roots are simple; it lies in an open subset of $\{a\in\calP^{2g}_\bbR \mid a(0)=1\}$; the $A$-periods of $\Theta_b$ are real and the real parts of its $B$-periods are determined by them; and $\BPer$ is smooth on $\calB$ and injective on $\calB_a$ for $g\geq2$.
The first two hold on $\calU^g_0$ by definition, the third by the discussion of $A$-cycles on $\calU^g_0$ above, and the last by Lemma~\ref{lem:b_omega smooth node}.
Lemma~\refinv{lem:whitham kernel}{2.2(iii)} treats a common root of $\calB_a$.
At nodal curves we do not consider common roots of $\calB_a$ other than the nodes themselves, so that case does not really arise in this paper.
Nonetheless, Lemma~\refinv{lem:whitham kernel}{2.2(iii)} extends to $\calU^g_0$ under one additional hypothesis, as follows.
Its proof uses that the single common root $\beta$ of $b_1$ and $b_2$ lies on $\bbS^1$, and so is not a root of $a\in\calH^g_0$; this is what lets the argument of part~(ii) apply.
On $\calU^g_0$ the polynomial $a$ may have roots on $\bbS^1$, so there we assume in addition that $\beta$ is not a root of $a$.

Since $\calB$ and $\BPer$ are defined on all of $\calU^g$ by Lemma~\ref{lem:b_omega smooth node}, we define the isoperiodic sets $\calU^g(a)$ and $\calU^g(a,\tau)$ as in Definition~\refinv{def:isoperiodic sets}{2.1}, and likewise $\calU^g_0(a)$ and $\calU^g_0(a,\tau)$ in $\calU^g_0$.
Then Theorems~\numinv{thm:Whitham deformation}{2.3} and~\refinv{thm:Willmore derivative}{2.5} hold for $a\in\calU^g_0$, with $\calU^g_0(a)$ and $\calU^g_0(a,\tau)$ in place of $\calH^g_0(a)$ and $\calH^g_0(a,\tau)$, under the same additional assumption that a common root of $\calB_a$ is not a root of $a$.
In particular, $\calU^g_0(a)$ and $\calU^g_0(a,\tau)$ are $3$- and $1$-dimensional manifolds, respectively.
We also define the normalized Willmore energy $\willmore$ by~\eqref{eq:willmore} for all $a\in\calU^g$.
At a nodal curve, the argument that the residue in the numerator of~\eqref{eq:willmore} is real goes through on the normalization, on which $\Theta_{b_1}$ and $\Theta_{b_2}$ have poles only over $\lambda=0,\infty$ by Lemma~\ref{lem:b_omega continuous}.

\section{Isoperiodic Sets at Nodal Curves}
\label{sec:isoperiodic nodes}

The results of this section are extensions of those of Section~\refinv{sec:isoperiodic sets}{2} to nodal curves, that is, to polynomials $a$ with double roots on $\bbS^1$.
The main difference is that the Whitham equations~\eqref{eq:whitham ab} no longer completely encode the tangent directions to the isoperiodic set.
They must be augmented with conditions that ensure the constancy of periods along $B$-cycles that pass through a node.
Under reasonable assumptions, the isoperiodic sets extend smoothly to such $a$.

The strategy is to control the kernel of the derivative of the period map.
Lemma~\ref{lem:whitham kernel node}(ii) and (iii), like Lemma~\refinv{lem:whitham kernel}{2.2(ii) and (iii)}, say that the kernel of $\dd\Phi$ is $3$-dimensional, and parametrize it (as much as possible) by $Q\in\calP^2_\bbR$.
These results hold under a nondegeneracy condition $D_j\neq0$ at each node $\alpha_j$ that has no analogue in Lemma~\refinv{lem:whitham kernel}{2.2}.
In the smooth case, a common root of $\calB_a$ on $\bbS^1$ is an impediment to the isoperiodic flow.
At a node, $\calB_a$ is forced to have a common root, but this is not an impediment to the isoperiodic flow if the common root is first order and $D_j \neq 0$.
Theorem~\ref{thm:Whitham deformation node} extends Theorem~\refinv{thm:Whitham deformation}{2.3}: the isoperiodic sets $\calU^g(a)$ and $\calU^g(a,\tau)$ remain submanifolds of dimensions $3$ and $1$ at the nodal curves.
The isoperiodic sets are transverse to the subset of nodal curves, and we describe how the nodes open, whether each double root splits into a conjugate pair off $\bbS^1$ or into two simple roots on $\bbS^1$.
Corollary~\ref{cor:Willmore derivative node} extends Theorem~\refinv{thm:Willmore derivative}{2.5}: $\dot{\willmore}$ vanishes at a nodal curve if and only if $D_n = 0$.
Finally, Lemma~\ref{lem:closure S} shows that only nodal curves with $D_j = 0$ at some node can lie in the closure of $\calS^g$, which suggests the connection between the two parts~(iii), and between Theorem~\refinv{thm:Willmore derivative}{2.5} and Corollary~\ref{cor:Willmore derivative node}.

The closest comparison is with Kewlin's thesis~\cite{Kewlin2015thesis}, which reproves the existence result of Jaggy~\cite{Jaggy1994} with Whitham flows instead of the inverse function theorem.
Kewlin works with the spectral curves of CMC tori, preserving the conditions at the Sym point as well as the periods, and opens a single double point away from the Sym point.
He finds curves with a suitable double point by combining its position with a Whitham flow of the normalization~\cite[Lemma~4.7]{Kewlin2015thesis}, much as we do in Lemma~\ref{lem:insert node}.
He uses the Whitham equations, as polynomial equations, only to control the flows that keep the double point closed.
To open it, he needs an ordinary differential equation, whose solution he constructs as a formal power series and whose convergence he proves separately~\cite[Section~4.5]{Kewlin2015thesis}.
We instead add the conditions from the nodal $B$-periods to the Whitham equations (Lemmas~\ref{lem:nodal B period derivative} and~\ref{lem:whitham kernel node}), which are linear conditions on the same polynomials.
The opening flows then lie in the same polynomial framework as all the others, and the whole isoperiodic set passes through the nodal curve as a manifold.
This also allows any number of double points.
Unlike Kewlin, and the inverse function theorem arguments of~\cite{Ercolani1993,Jaggy1994}, we do not consider the preservation of the two further conditions needed for a torus, a common root of $\Theta_{b_1},\Theta_{b_2}$ at a Sym point $\lambda_0\in\bbS^1$ and integral periods along a path between the two points over $\lambda_0$ (conditions (P4) and (P5) of \refinv{sec:differentials}{Section~1.2}), although this should be feasible within our framework.

We first show that the Whitham equations still characterize the directions that preserve the non-nodal periods, and then compute the derivative of the nodal periods.
Let us consider a path $(a_t,b_t)$ in $\calB$ that has a nodal curve only at time $t=0$.
That is, $a_0 \in \calU^g_n$ but $a_t \in \calU^g_0$ for $t \neq 0$.
Call the nodes $\alpha_1, \dots, \alpha_n \in \bbS^1$.
We assume that the path opens each node to first order, that is, $\dot{a}(\alpha_j) \neq 0$ for $j = 1,\dots,n$.
We choose a basis of homology as constructed before Remark~\ref{rem:normalized differentials}; in particular, for each node $j=1,\dots,n$ we take the neck cycle $A_j$ as an $A$-cycle in the basis, and the $B_j$-cycle is a path from $\alpha_j^-$ to $\alpha_j^+$.
Since we are considering a single path $(a_t,b_t)$, we may suppose that the cycles are constant in the $\lambda$-plane, in contrast to that construction, where we needed cycles that vary over a neighborhood of $a \in \calU^g$.
Let $p_{\alpha_j}(\lambda) = \ci \alpha_j^{-1/2}(\lambda - \alpha_j)$ and $p(\lambda) = p_{\alpha_1}(\lambda)\cdots p_{\alpha_n}(\lambda)$.
At $t=0$ we have $a = p^2 \tilde{a}$ and $b = p \tilde{b}$.
On the normalization $\Sigma_{\tilde{a}}$, the cycles $A_{n+1},\dots,A_g,B_{n+1},\dots,B_g$ pull back to a homology basis of $\Sigma_{\tilde{a}}$ and $A_j$ pulls back to a loop around one of the points over $\alpha_j$.
We know that $\dot{\Theta}_{b_t}$ pulled back to the normalization may have double poles over the roots of $\lambda a(\lambda)$.
At the branch points, the residues are necessarily zero, and at the points over a node, the residues are negatives of each other.
Because $b_t \in \calB_{a_t}$, its $A$-periods are always zero. 
Therefore the $A$-periods of $\dot{\Theta}_{b_t}$ are also zero, and hence the residues are zero.
Together this proves that $\dot{\Theta}_{b_t}$ is a differential of the second kind.

Let us first show how to control the periods of one differential at a nodal curve.
Specifically, we consider the non-nodal period map on $\calB$
\[
(a,b) \mapsto \left( \int_{B_j} \Theta_{b} \right)_{j=n+1,\dots,g}.
\]
By definition, $(\dot{a}_0,\dot{b}_0)$ belongs to the kernel of the derivative of this map if and only if
\[
0
= \frac{\dd}{\dd t}\left.\int_{B_j} \Theta_{b_t} \right|_{t=0} 
= \int_{B_j} \dot{\Theta}_{b_0}
\qquad\text{ for }\qquad
j=n+1,\dots,g.
\]
As $\dot{\Theta}_{b_0}$ is of the second kind (as shown above) and its periods are zero, it is exact.
The primitive has at most simple poles at the roots of $\lambda a(\lambda)$.
By the symmetries used at the start of the proof of Lemma~\refinv{lem:whitham kernel}{2.2}, it follows that
\[
\dot{\Theta}_{b_0} = \dd \left( \ci \frac{c(\lambda)}{\nu} \right)
\]
for $c \in \calP^{g+1}_\bbR$.
Conversely, if $\dot{\Theta}_{b_0}$ is an exact differential, then it follows that $(\dot{a},\dot{b})$ belongs to the kernel of the derivative.
Computing the $t$ and $\lambda$ derivatives, we arrive at the same Whitham equation as~\eqref{eq:whitham ab}.

We will need in several places certain relations among these polynomials at the nodes.
Take the Whitham equation~\eqref{eq:whitham ab}, differentiate once, and evaluate at $\alpha_j$:
\begin{equation}
\label{eq:whitham eval node}
\dot{a}(\alpha_j)b'(\alpha_j)
= \ci\alpha_j a''(\alpha_j)c(\alpha_j).
\end{equation}
Differentiating twice and then evaluating at $\alpha_j$ gives
\begin{equation}
\label{eq:whitham deriv eval node}
2a''(\alpha_j)\dot{b}(\alpha_j) - 2\dot{a}'(\alpha_j)b'(\alpha_j) - \dot{a}(\alpha_j)b''(\alpha_j)
= -\ci(\lambda a)'''(\alpha_j)\,c(\alpha_j).
\end{equation}

Now we handle two differentials along the path $a_t \in \calU^g$.
Suppose that $(\dot{a},\dot{b}_1,\dot{b}_2)$ lies in the kernel of the derivative of 
\[
(a,b_1,b_2) \mapsto \left( \int_{B_j} \Theta_{b_1}, \int_{B_j} \Theta_{b_2} \right)_{j=n+1,\dots,g}.
\] 
As in the proof of Lemma~\refinv{lem:whitham kernel}{2.2}, we have~\eqrefinv{eq:whitham antisymmetrize}{2.4}, from which we conclude that $a$ divides $\ci \lambda a' (b_2 c_1 - b_1 c_2)$.
Observe that $\tilde{a}$ is coprime to $\lambda a'$, so it divides $b_2 c_1 - b_1 c_2$.
However, $a = p^2\tilde{a}$ shares the factor $p$ with $a'$, so we cannot directly make the same conclusion for $p$.
Instead,~\eqref{eq:whitham eval node} tells us that 
\[
(b_2 c_1 - b_1 c_2)'(\alpha_j)
= (b_2 c_1' - b_1 c_2' + b_2' c_1 - b_1' c_2)(\alpha_j)
= 0.
\]
In other words, $\alpha_j$ is a root of $b_2 c_1 - b_1 c_2$ of at least second order, and thus $p^2$ does divide it.
This means that~\eqref{eq:whitham Q} also holds in this case.
In summary, the Whitham equations~\eqref{eq:whitham ab} and~\eqref{eq:whitham Q} characterize exactly the directions that preserve the $A$-periods and the $B_j$-periods for $j = n+1,\dots,g$, that is, the non-nodal $B$-periods.

Next, we develop a formula for how the nodal $B_j$-periods change.
We do so under the condition that $\dot{a}(\alpha_j) \neq 0$ at every double root $\alpha_j$.
We will see that this implies that the double root is deformed to two simple roots along the path $a_t$.
This is generically true, so the restriction is a technicality, not a fundamental obstacle; we remove it by linearity in the proof of Lemma~\ref{lem:whitham kernel node}.

\begin{lemma}
\label{lem:nodal B period derivative}
Let $a\in\calU^g_n$ have double roots $\alpha_1,\dots,\alpha_n \in \bbS^1$, and let $(a_t,b_t)$ be a smooth path in $\calB$ with $a_0=a$.
Suppose $\dot{\Theta}_{b_0} = \dd\left(\ci\frac{c}{\nu}\right)$ for some $c\in\calP^{g+1}_\bbR$ and $\dot{a}(\alpha_j) \neq 0$ for all $j=1,\dots,n$.
Let $\Omega_j = u_j\,\frac{\dd\lambda}{\nu}$ be the normalized differential of the third kind corresponding to the node $\alpha_j$, as in Remark~\ref{rem:normalized differentials}.
Then for $j=1,\dots,n$,
\[
\left.\frac{\dd}{\dd t}\int_{B_j} \Theta_{b_t} \right|_{t=0}
= - 4 \pi\ci \frac{u_{j}(\alpha_j)}{\alpha_j^2 a''(\alpha_j)} \left[ \dot{b}(\alpha_j) + \frac{b'(\alpha_j)}{a''(\alpha_j)}\dot{a}'(\alpha_j) - 2\ci\alpha_j c'(\alpha_j) + \ci c(\alpha_j) \right].
\]
\end{lemma}
\begin{proof}
Above, we were able to pass the derivative inside the integral using the dominated convergence theorem. 
This is because the cycles $B_j$ were disjoint from the nodes for $j = n+1,\dots,g$.
This fails for $t=0$ and $j=1,\dots,n$, since then the poles of $\dot{\Theta}_{b_0}$ lie on the path of integration.
However, $\BPer$ is a smooth function on $\calB$ by Lemma~\ref{lem:b_omega smooth node}, so it suffices to take a limit
\[
\left.\frac{\dd}{\dd t}\int_{B_j} \Theta_{b_t} \right|_{t=0}
= \lim_{t \to 0} \int_{B_j} \dot{\Theta}_{b_t}.
\]
We have already argued that $\dot{\Theta}_{b_t}$ is a differential of the second kind with zero $A$-periods.
Equation~\eqref{eq:B period residue} uses only these properties of $\Theta_b$, so it applies equally to $\dot{\Theta}_{b_t}$ and its primitive $\dot{q}_{b_t}$:
\begin{align*}
\int_{B_j} \dot{\Theta}_{b_t}
= - 2\pi\ci \Res \left( \Omega_{j,t} \dot{q}_{b_t} \right)
= 2\pi\ci \Res \left( \Omega_{j,t} \left[ \ci \frac{c}{\nu_t} - \dot{q}_{b_t} \right] \right).
\end{align*}
For the second equality, we have used that adding $\ci \frac{c}{\nu_t} \Omega_{j,t}$ does not change the residue because it is a meromorphic differential.
This modification is a convenient trick that helps us to expand the right-hand side as a function in $t$ around $t=0$.
We introduce the polynomial $P_t(\lambda)$ defined by
\[
\dd \left[ \ci \frac{c}{\nu_t} - \dot{q}_{b_t} \right] = -\frac{P_t\,\dd\lambda}{2\lambda a_t\nu_t},
\qquad 
P_t=2a_t \dot{b}_t- \dot{a}_t b_t-\ci\big(2\lambda a_tc'-(\lambda a_t'+a_t)c\big).
\]
Note that this is not the Whitham equation~\eqref{eq:whitham ab} because $c$ is fixed along the path $a_t$, although it is obviously related.
At $t=0$, however, it is the Whitham equation, since $\dot{\Theta}_{b_0} = \dd\left( \ci \frac{c}{\nu_0} \right)$, and so $P_0 \equiv 0$.

Here $\Res$ denotes the sum of the residues at all poles, which are exactly the roots of $\nu_t$ (including $\lambda = \infty$).
Near $\alpha_t$, a root of $\nu_t = \sqrt{\lambda a_t(\lambda)}$, we may expand
\[
\dd\left[ \ci \frac{c}{\nu_t} - \dot{q}_{b_t} \right] = \frac{P_t(\alpha_t)}{(\lambda a_t)'(\alpha_t)}\,\dd\Big(\frac1{\nu_t}\Big)+\text{holomorphic near }\alpha_t,
\]
so that
\[
\Res_{\alpha_t} \left( \Omega_{j,t} \left[ \ci \frac{c}{\nu_t} - \dot{q}_{b_t} \right] \right)
= u_{j,t}(\alpha_t) \frac{P_t(\alpha_t)}{(\lambda a_t)'(\alpha_t)} \Res_{\alpha_t} \frac{\dd\lambda}{\lambda a_t(\lambda)}
= \frac{2 u_{j,t}(\alpha_t) P_t(\alpha_t) }{[(\lambda a_t)'(\alpha_t)]^2}.
\]
At a root of $\lambda a_t(\lambda)$ that does not tend to a double root of $a$, so that it tends to a simple root of $\lambda a(\lambda)$, every factor in this expression remains bounded and the denominator remains nonzero.
Then $P_0 \equiv 0$ forces these residues to tend to zero as $t \to 0$.
If $\alpha_t$ tends to a double root $\alpha_k$ of $a(\lambda) = a_0(\lambda)$, then the denominator tends to zero and we will have to balance it against the vanishing of $P_t$.
Specifically, using the continuity of $\Omega_{j,t}$ in $t$ from Remark~\ref{rem:normalized differentials}, we will calculate
\[
\lim_{t \to 0} \Res_{\alpha_t} \left( \Omega_{j,t} \left[ \ci \frac{c}{\nu_t} - \dot{q}_{b_t} \right]  \right)
= \frac{2 u_{j}(\alpha_k) \lim_{t\to 0} t^{-1} P_t(\alpha_t) }{\alpha_k^2 \lim_{t\to 0} t^{-1} [a_t'(\alpha_t)]^2}.
\]
There are two roots of $a_t$ that converge to $\alpha_k$; one is $\alpha_t$ and call the other $\alpha^*_t$.
We write 
\[
a_t(\lambda) 
= f_t(\lambda)\tilde{a}_{t}(\lambda)
\qquad\text{for}\qquad
f_t(\lambda) = p_{\alpha_t}(\lambda)p_{\alpha^*_t}(\lambda) = R_t(\lambda - \alpha_t)(\lambda - \alpha^*_t),
\]
where $R_t = -(\alpha_t\alpha^*_t)^{-1/2}$, with the square root chosen continuously so that $R_0 = -\alpha_k^{-1}$ and $f_0 = p_{\alpha_k}^2$.
Because it is a symmetric function of the roots $\alpha_t$ and $\alpha^*_t$, the factor $f_t$ is a smooth function of the coefficients of $a_t$, even when the roots come together, and therefore a smooth function of $t$.
In particular, $f_0'$ has a root at $\alpha_k$ and hence $f_t'(\alpha_k) = O(t)$.
We compute, using this fact in the last line:
\begin{align*}
\lim_{t \to 0} t^{-1} [a_t'(\alpha_t)]^2
&= \lim_{t \to 0} t^{-1} f_t'(\alpha_t)^2\, \tilde{a}_{t}(\alpha_t)^2 
= \tilde{a}(\alpha_k)^2 \lim_{t \to 0}  t^{-1} R_t^2[(\alpha_k - \alpha_t) - (\alpha_k - \alpha^*_t)]^2 \\
&= \tilde{a}(\alpha_k)^2 \lim_{t \to 0}  t^{-1} R_t^2\left[ [(\alpha_k - \alpha_t) + (\alpha_k - \alpha^*_t)]^2 - 4(\alpha_k - \alpha_t) (\alpha_k - \alpha^*_t) \right] \\
&= \tilde{a}(\alpha_k)^2 \lim_{t \to 0}  t^{-1} [f_t'(\alpha_k)]^2
- 4 R_0\tilde{a}(\alpha_k)^2 \lim_{t \to 0}  t^{-1} f_t(\alpha_k) \\
&= 0 - 4R_0\tilde{a}(\alpha_k)^2 \lim_{t \to 0}  t^{-1} f_t(\alpha_k).
\end{align*}
On the other hand, we can expand $a_t$ in $t$ and evaluate at $\alpha_k$:
\[
f_t(\alpha_k)\tilde{a}_t(\alpha_k)
= a_t(\alpha_k)
= a(\alpha_k) + t \dot{a}(\alpha_k) + O(t^2) 
= 0 + t \dot{a}(\alpha_k) + O(t^2).
\]
Putting these together yields
\[
\lim_{t \to 0} t^{-1} [a_t'(\alpha_t)]^2
= - 4R_0\tilde{a}(\alpha_k)^2 \lim_{t \to 0} t^{-1} \frac{t\dot{a}(\alpha_k)}{\tilde{a}_t(\alpha_k)}
= - 4R_0\tilde{a}(\alpha_k) \dot{a}(\alpha_k)
= - 2a''(\alpha_k) \dot{a}(\alpha_k),
\]
where the last equality is $a''(\alpha_k) = f_0''(\alpha_k)\tilde{a}(\alpha_k) = 2R_0\tilde{a}(\alpha_k)$.
This is nonzero by the assumption that $\alpha_k$ is a double root of $a$ and $\dot{a}(\alpha_k) \neq 0$.
For the numerator, $P_t$ is a polynomial of fixed degree in $\lambda$ whose coefficients are smooth in $t$ and vanish at $t=0$, so we Taylor expand each coefficient to first order in $t$:
\begin{align*}
\lim_{t \to 0} t^{-1} P_t(\lambda)
= 2a \ddot{b} + \dot{a} \dot{b} - \ddot{a} b - \ci\big(2\lambda \dot{a} c'-(\lambda \dot{a}'+ \dot{a})c\big).
\end{align*}
Therefore
\begin{align*}
\lim_{t \to 0} t^{-1} P_t(\alpha_t)
&= 0 + \dot{a}(\alpha_k) \dot{b}(\alpha_k) - 0 - \ci\big(2\alpha_k \dot{a}(\alpha_k) c'(\alpha_k)-(\alpha_k \dot{a}'(\alpha_k)+ \dot{a}(\alpha_k))c(\alpha_k)\big).
\end{align*}
In total we have
\[
\lim_{t \to 0} \Res_{\alpha_t} \left( \Omega_{j,t} \left[ \ci \frac{c}{\nu_t} - \dot{q}_{b_t} \right] \right)
= - \frac{u_{j}(\alpha_k)}{\alpha_k^2 a''(\alpha_k)} \left[ \dot{b}(\alpha_k) - 2\ci\alpha_k c'(\alpha_k) + \ci \alpha_k \frac{\dot{a}'(\alpha_k)}{\dot{a}(\alpha_k)}c(\alpha_k) + \ci c(\alpha_k) \right].
\]
For $k \neq j$ we have $u_j(\alpha_k) = 0$ by Remark~\ref{rem:normalized differentials}, so only the two roots that converge to $\alpha_j$ contribute, and they contribute equally. 
Therefore we have derived the formula
\[
\left.\frac{\dd}{\dd t}\int_{B_j} \Theta_{b_t} \right|_{t=0}
= - 4 \pi\ci \frac{u_{j}(\alpha_j)}{\alpha_j^2 a''(\alpha_j)} \left[ \dot{b}(\alpha_j) - 2\ci\alpha_j c'(\alpha_j) + \ci \alpha_j \frac{\dot{a}'(\alpha_j)}{\dot{a}(\alpha_j)}c(\alpha_j) + \ci c(\alpha_j) \right].
\]
Only the third term in the bracket is not linear in the vector $(\dot{a},\dot{b})$.
Substituting~\eqref{eq:whitham eval node} into the above removes the undesirable term and gives the formula for the derivative of the period.
\end{proof}

At a nodal curve the map~\eqref{eq:period map} is no longer defined, since the $B$-period of $\Theta_b$ along a cycle through a node diverges unless $b$ has a root there.
We therefore restrict to differentials in $\calB$ and consider
\[
\Phi=(\BPer,\BPer): \calB\oplus\calB \to \bbR^{2g},
\qquad
\calB\oplus\calB = \calU^g \times (\calB_a \oplus \calB_a),
\]
which is essentially~\eqref{eq:period map} extended to $a \in \calU^g \supset \calU^g_0$ but restricted to $b_k \in \calB_a \subset \calP^{g+1}_\bbR$.
It is smooth by Lemma~\ref{lem:b_omega smooth node}.

\begin{lemma}
\label{lem:whitham kernel node}
Let $a\in\calU^g_n$, with double roots $\alpha_1,\dots,\alpha_n \in \bbS^1$, and let $b_1,b_2\in\calB_a$ be linearly independent.
Assume that each $\alpha_j$ is a common root of $\calB_a$ of first order and that $\calB_a$ has no other common roots.
Write $a = p^2 \tilde{a}$ and $b_k = p \tilde{b}_k$ for $p \in \calP^n_\bbR$, and define
\[
D_j
=b_1'(\alpha_j)b_2''(\alpha_j)-b_2'(\alpha_j)b_1''(\alpha_j)
=2p'(\alpha_j)^2 \left[ \tilde{b}_1(\alpha_j)\tilde{b}_2'(\alpha_j)-\tilde{b}_2(\alpha_j)\tilde{b}_1'(\alpha_j) \right]
\]
for $j=1,\dots,n$.
We describe the kernel of $\dd\Phi$ at $(a,b_1,b_2)$.
\begin{enumerate}
\item
A tangent vector $(\dot{a},\dot{b}_1,\dot{b}_2)$ belongs to the kernel if and only if there are $c_1,c_2\in\calP^{g+1}_\bbR$ that solve the \emph{Whitham equations}:
\begin{align}
\label{eq:whitham ab node}
2p \tilde{a}\dot{b}_k - \dot{a}\tilde{b}_k = \ci(2\lambda p\tilde{a} c_k' - (2\lambda p' \tilde{a} + \lambda p\tilde{a}'+ p\tilde{a})c_k)
\end{align}
and
\begin{align}
\Delta_{k,j}
:= \dot{b}_k(\alpha_j) + \frac{b'_k(\alpha_j)}{a''(\alpha_j)}\dot{a}'(\alpha_j) - 2\ci\alpha_j c_k'(\alpha_j) + \ci c_k(\alpha_j) = 0
\end{align}
for $k=1,2$ and $j=1,\dots,n$.
Then there exists $Q\in\calP^2_\bbR$ such that the \emph{compatibility equation} holds:
\begin{align}
\label{eq:whitham Q node}
\tilde{b}_2c_1 - \tilde{b}_1c_2 = p\tilde{a}Q.
\end{align}
The polynomials $c_1,c_2$ and $Q$ are uniquely determined by $(\dot{a},\dot{b}_1,\dot{b}_2)$.
Further
\begin{equation}
\label{eq:node dot a}
2\ci\alpha_j\,a''(\alpha_j)^2\,Q(\alpha_j)
=
D_j\, \dot{a}(\alpha_j)
\end{equation}
for $j=1,\dots,n$.

\item
If $D_j \neq0$ for $j=1,\dots,n$, then $(\dot{a},\dot{b}_1,\dot{b}_2)\mapsto Q$ is an isomorphism from the kernel onto $\calP^2_\bbR$.
In particular the kernel is three-dimensional.

\item
If $D_j\neq0$ for $j=1,\dots,n-1$ but $D_n = 0$, then the kernel is three-dimensional and $Q$ maps it surjectively onto $\{Q\in\calP^2_\bbR\mid Q(\alpha_n)=0\}$.
The vectors with $Q=0$ form a line, along which $\dot{a}(\alpha_n)\neq0$.
\end{enumerate}
\end{lemma}

Before we launch into the proof, let us provide some commentary on the definitions embedded in this statement.
The condition that there is a common root of $\calB_a$ of first order implies that $\tilde{b}_1,\tilde{b}_2$ do not have a common root at $\alpha_j$.
Together with the second hypothesis, $\tilde{b}_1$ and $\tilde{b}_2$ have no common roots at all.
For $g=n$ the polynomials $\tilde{b}_k$ are linear, so $D_j\neq0$ holds automatically and (ii) applies.
Because $g \geq n \geq 1$, there is no $g=0$ case to consider, as there was in Lemma~\refinv{lem:whitham kernel}{2.2}.
Under a change of basis $D_j$ is multiplied by the determinant, so whether $D_j$ vanishes does not depend on the basis of $\calB_a$.
For a basis with $b_1'(\alpha_j)\neq0$, l'Hôpital's rule at the simple common root $\alpha_j$ of $b_1,b_2$ gives $D_j=2b_1'(\alpha_j)^2f'(\alpha_j)$, where $f=b_2/b_1$ is the winding number function recalled in Section~\ref{sec:background}.
So the condition $D_j \neq 0$ says that $\frac{\dd}{\dd\lambda}\frac{b_2}{b_1}$ does not vanish at $\alpha_j$, which is the condition appearing in the second part of Theorem~\refinv{thm:Willmore derivative}{2.5}.
Alternatively, for each node $\alpha_j$ there is $b_3\in\calB_a$, unique up to scale, with $b_3'(\alpha_j)=0$, since $b\mapsto b'(\alpha_j)$ vanishes on a line in $\calB_a$.
Note, this gives different polynomials at different nodes.
Taking $b_2=b_3$ gives $D_j=b_1'(\alpha_j)b_3''(\alpha_j)$, so $D_j=0$ if and only if $b_3$ has a root of order at least three at $\alpha_j$.

\begin{proof}
The first two sentences of (i) are simply a summary of Lemma~\ref{lem:nodal B period derivative} and the discussion preceding it.
Specifically, there exists $(c_1,c_2,Q)$ solving~\eqref{eq:whitham ab} and~\eqref{eq:whitham Q} if and only if the derivatives of the $A$-periods and the $B_{n+1},\dots,B_g$-periods are zero.
After substituting in the factorizations $a = p^2 \tilde{a}$ and $b = p \tilde{b}$ and removing a factor of $p$ from both sides, we arrive at~\eqref{eq:whitham ab node} and~\eqref{eq:whitham Q node}.
The quantity $\Delta_{k,j}$ is exactly the bracket in Lemma~\ref{lem:nodal B period derivative}.
Since $u_j(\alpha_j)\neq0$ by Remark~\ref{rem:normalized differentials}, it is zero if and only if the derivative of the $B_j$-period of $\Theta_{b_k}$ is zero.
The polynomials $c_1,c_2$ and $Q$ are uniquely determined by $(\dot{a},\dot{b}_1,\dot{b}_2)$, by the same argument as Lemma~\refinv{lem:whitham kernel}{2.2}.
That argument does not use that the roots of $a$ are simple: the map $c\mapsto 2\lambda ac'-(\lambda a'+a)c = 2\nu^3 \frac{\dd}{\dd\lambda}(c/\nu)$ of~\eqrefinv{eq:c injective}{2.5} is injective on $\calP^{g+1}_\bbR$ because $c/\nu$ is constant only for $c=0$, and multiplication by $a$ is injective.

The vanishing of $\Delta_{k,j}$ imposes some conditions on $(c_1,c_2,Q)$ at $\alpha_j$, which lead to~\eqref{eq:node dot a}.
Evaluation of~\eqref{eq:whitham Q node} and its derivative at a node produces
\begin{gather*}
\tilde{b}_2 c_1 - \tilde{b}_1 c_2 \Big|_{\lambda=\alpha_j}
= p\tilde{a}Q\Big|_{\lambda=\alpha_j} = 0, \\
\tilde{b}'_2 c_1 - \tilde{b}'_1 c_2 + \tilde{b}_2 c'_1 - \tilde{b}_1 c'_2 \Big|_{\lambda=\alpha_j}
= (p\tilde{a}Q)'\Big|_{\lambda=\alpha_j} = p'\tilde{a}Q\Big|_{\lambda=\alpha_j}.
\end{gather*}
In terms of $\tilde{a}$ and $\tilde{b}_k$, and using~\eqref{eq:whitham eval node}, equation~\eqref{eq:whitham deriv eval node} reads
\[
2p'\tilde{a}\dot{b}_k - \dot{a}'\tilde{b}_k - \dot{a}\tilde{b}_k' \Big|_{\lambda=\alpha_j}
= - \ci\big( 3p'(\lambda\tilde{a})' + 2\lambda p''\tilde{a} \big)c_k \Big|_{\lambda=\alpha_j}.
\]
Antisymmetrizing gives
\begin{align*}
2p'\tilde{a}(\tilde{b}_2\dot{b}_1 - \tilde{b}_1\dot{b}_2) \Big|_{\lambda=\alpha_j}
&= 0 + \dot{a}(\tilde{b}_2\tilde{b}'_1 - \tilde{b}_1\tilde{b}'_2) - \ci\big(3p'(\lambda\tilde{a})' + 2\alpha_j p''\tilde{a}\big)(\tilde{b}_2 c_1 - \tilde{b}_1 c_2) \Big|_{\lambda=\alpha_j} \\
&= 2\ci\alpha_j p'\tilde{a}(c_2\tilde{b}'_1 - c_1\tilde{b}'_2) - 0 \Big|_{\lambda=\alpha_j},
\end{align*}
using~\eqref{eq:whitham eval node} for the first term and the first condition above for the second.
Since $b'_k/a'' = \tilde{b}_k/(2p'\tilde{a})$ at $\alpha_j$, antisymmetrizing $\Delta_{k,j}$ gives
\begin{align*}
2p'\tilde{a}(\tilde{b}_2 \Delta_{1,j} - \tilde{b}_1 \Delta_{2,j}) \Big|_{\lambda=\alpha_j}
&= 2p'\tilde{a}(\tilde{b}_2\dot{b}_1 - \tilde{b}_1\dot{b}_2)
+ 0 - 4\ci\alpha_j p'\tilde{a}(\tilde{b}_2 c'_1 - \tilde{b}_1 c'_2)
+ 0 \Big|_{\lambda=\alpha_j} \\
&= 2\ci\alpha_j p'\tilde{a} \left[ (c_2\tilde{b}'_1 - c_1\tilde{b}'_2)
- 2(\tilde{b}_2 c'_1 - \tilde{b}_1 c'_2) \right]\Big|_{\lambda=\alpha_j} \\
&= 2\ci\alpha_j p'\tilde{a} \left[ (\tilde{b}'_2 c_1 - \tilde{b}'_1 c_2)
- 2p'\tilde{a}Q \right]\Big|_{\lambda=\alpha_j},
\end{align*}
where the last equality substitutes the second condition above for $\tilde{b}_2 c'_1 - \tilde{b}_1 c'_2$.
Therefore we have
\[
\tilde{b}'_2(\alpha_j) c_1(\alpha_j) - \tilde{b}'_1(\alpha_j) c_2(\alpha_j)
= 2p'(\alpha_j)\tilde{a}(\alpha_j)Q(\alpha_j).
\]
Since $a''(\alpha_j)\neq0$, equation~\eqref{eq:whitham eval node} gives $c_k(\alpha_j)=\dot{a}(\alpha_j)b_k'(\alpha_j)/(\ci\alpha_ja''(\alpha_j))$.
Substituting this on the left, and using $b_k'=p'\tilde{b}_k$ and $a''=2{p'}^2\tilde{a}$ at $\alpha_j$, gives
\[
\frac{\dot{a}(\alpha_j)D_j}{2\ci\alpha_j p'(\alpha_j)a''(\alpha_j)}
= \frac{a''(\alpha_j)Q(\alpha_j)}{p'(\alpha_j)},
\]
which is~\eqref{eq:node dot a}.

Lemma~\ref{lem:nodal B period derivative}, which we used for the first claim of (i), carries the assumption that $\dot{a}(\alpha_j)\neq0$ for $j=1,\dots,n$.
We remove this assumption by linearity.
Consider the space of tangent vectors $(\dot{a},\dot{b}_1,\dot{b}_2)$ that are solutions of~\eqref{eq:whitham ab node}.
By uniqueness, $c_k$ is linear in $(\dot{a},\dot{b}_1,\dot{b}_2)$.
Therefore the derivative of the $B_j$-period of $\Theta_{b_k}$ and the right-hand side of the formula in Lemma~\ref{lem:nodal B period derivative} are both linear functionals on this space, and they agree outside the union of the subspaces $\{\dot{a}(\alpha_j)=0\}$, $j=1,\dots,n$.
If each of these subspaces is proper, the complement of their union is dense and the two functionals agree everywhere.
To see that each is a proper subspace, take $u\in\calP^n_\bbR$ with $u(\alpha_j)\neq0$ for $j=1,\dots,n$.
A product of factors $p_\beta$ with $\beta\in\bbS^1$ distinct from the nodes suffices.
Then $c_k=u\tilde{b}_k$ solves~\eqref{eq:whitham ab node} with
\[
\dot{a} = \ci u\,\frac{\lambda a'-ga}{p},
\qquad
\dot{b}_k = \frac{\ci}{2}\big(2\lambda c_k'-(g+1)c_k\big).
\]
Since $\lambda a'-ga$ vanishes at the double roots of $a$, it is divisible by $p$, and $\dot{a}\in\calP^{2g}_\bbR$ because a quotient of elements of $\calP^m_\bbR$ and $\calP^n_\bbR$ that is a polynomial lies in $\calP^{m-n}_\bbR$.
This vector has $\dot{a}(\alpha_j) = 2\ci\alpha_j p'(\alpha_j)\tilde{a}(\alpha_j)u(\alpha_j)\neq0$.
Adding real multiples of $(2a,b_1,b_2)$, which solves~\eqref{eq:whitham ab} with $c_k=0$ and has $\dot{a}(0)=2$, and of the trivial solution $\ci\big(\lambda a'-ga,\lambda b_1'-\tfrac{g+1}{2}b_1,\lambda b_2'-\tfrac{g+1}{2}b_2\big)$, which has $c_k=b_k$ and $\dot{a}(0)=-\ci g$, achieves $\dot{a}(0)=0$, and does not change $\dot{a}(\alpha_j)$ since $a$ and $\lambda a'-ga$ vanish at $\alpha_j$.
Therefore this vector lies outside of all the subspaces $\{\dot{a}(\alpha_j)=0\}$.
This completes the proof of (i).

Since $\dim(\calB\oplus\calB) - 2g = (2g-1) + 4 - 2g = 3$, the dimension of $\ker \dd\Phi$ is at least $3$.
To prove (ii) we will show that the map from the kernel to $Q \in \calP^2_\bbR$ is injective.
It will then follow that the dimension of the kernel is at most $\dim \calP^2_\bbR = 3$ and therefore exactly $3$.
Assume therefore that $(\dot{a},\dot{b}_1,\dot{b}_2) \in \ker \dd\Phi$ and that the corresponding $Q = 0$.
By assumption, $\tilde{b}_1$ and $\tilde{b}_2$ are coprime.
Hence the solutions $(c_1,c_2)$ of~\eqref{eq:whitham Q node} are $v (\tilde{b}_1,\tilde{b}_2)$ for $v \in \calP^n_\bbR$.
Since $D_j \neq 0$, equation~\eqref{eq:node dot a} gives $\dot{a}(\alpha_j)=0$, and then~\eqref{eq:whitham eval node} shows that $(c_1,c_2)$ have a root at $\alpha_j$.
Thus $v(\alpha_j) = 0$ for $j=1,\dots,n$ implies that $v$ is a real multiple of $p$, and $(c_1,c_2) = r (b_1,b_2)$ for $r \in \bbR$.

The solutions of the Whitham equations with $(c_1,c_2) = r (b_1,b_2)$ are
\[
(\dot{a},\dot{b}_1,\dot{b}_2)
= r\ci\Big(\lambda a'-ga,\lambda b_1'-\tfrac{g+1}{2}b_1,\lambda b_2'-\tfrac{g+1}{2}b_2\Big) + w (2p\tilde{a},\tilde{b}_1,\tilde{b}_2)
\]
for some $w \in \calP^n_\bbR$ (compare~\cite[Lemma~4.6]{Carberry2019}).
Imposing $\Delta_{k,j}=0$ yields
\[
0 = \Delta_{k,j}
= r\ci\big(\alpha_j b_k' + \alpha_j b_k' - 2\alpha_j b_k' + 0\big) + \big(w\tilde{b}_k + w\tilde{b}_k\big) \Big|_{\lambda=\alpha_j}
= 2w(\alpha_j)\tilde{b}_k(\alpha_j).
\]
Since $\tilde{b}_1$ and $\tilde{b}_2$ have no common root, $w(\alpha_j)=0$ for $j=1,\dots,n$.
Again, we conclude that $w = s p$ for some $s \in \bbR$.
Hence $(\dot{a},\dot{b}_1,\dot{b}_2) = r\ci\big(\lambda a'-ga,\lambda b_1'-\tfrac{g+1}{2}b_1,\lambda b_2'-\tfrac{g+1}{2}b_2\big) + s(2a,b_1,b_2)$.
Since $a(0)=1$, evaluating the first component at $\lambda=0$ gives
\[
0 = \dot{a}(0) = -\ci g r + 2s.
\]
As $g\geq1$, both $r,s\in\bbR$ vanish, and therefore $(\dot{a},\dot{b}_1,\dot{b}_2)=0$.
This proves that the map from the kernel to $Q$ is injective, and hence (ii).

For (iii), suppose that $D_j \neq 0$ for $j=1,\dots,n-1$ but $D_n = 0$.
Since $\tilde{b}_1$ and $\tilde{b}_2$ have no common root, we assume $\tilde{b}_1(\alpha_n) \neq 0$ without loss of generality.
Since $a''(\alpha_n)\neq0$, equation~\eqref{eq:node dot a} at $\alpha_n$ shows that $Q(\alpha_n) = 0$.
By the reality condition, $\alpha_n^{-1}Q(\alpha_n)$ is real, so this is a single real condition and $\{Q\in\calP^2_\bbR \mid Q(\alpha_n)=0\}$ is two-dimensional.
So $Q$ alone cannot be injective on the kernel.
Instead we show that the map
\[
(\dot{a},\dot{b}_1,\dot{b}_2) \mapsto \big(\alpha_n^{-(g+1)/2}c_1(\alpha_n), Q\big)
\in \bbR \times \{Q\in\calP^2_\bbR \mid Q(\alpha_n)=0\}
\]
is injective on the kernel, where $\alpha_n^{-(g+1)/2}c_k(\alpha_n)$ is real by the reality condition $c \in \calP^{g+1}_\bbR$ for any fixed choice of the square root.

Suppose therefore that $c_1(\alpha_n) = 0$ and $Q = 0$.
As in (ii), $(c_1,c_2) = v(\tilde{b}_1,\tilde{b}_2)$ with $v(\alpha_j) = 0$ for $j=1,\dots,n-1$.
The only step that changes is at $\alpha_n$, where $D_n\neq0$ is no longer available.
Instead, we have directly $v(\alpha_n)\tilde{b}_1(\alpha_n) = c_1(\alpha_n) = 0$, and $\tilde{b}_1(\alpha_n) \neq 0$ gives $v(\alpha_n) = 0$.
Hence $v$ is again a real multiple of $p$, and the rest of the proof of (ii) applies verbatim.
Since the kernel has dimension at least $3$, this injective map is an isomorphism onto the $3$-dimensional space $\bbR \times \{Q\in\calP^2_\bbR \mid Q(\alpha_n)=0\}$.
In particular $Q$ maps the kernel onto $\{Q\in\calP^2_\bbR \mid Q(\alpha_n)=0\}$, and the vectors with $Q=0$ form a line on which $c_1(\alpha_n)\neq0$.
By~\eqref{eq:whitham eval node}, $\dot{a}(\alpha_n)b_1'(\alpha_n) = \ci\alpha_n a''(\alpha_n)c_1(\alpha_n)$, so $\dot{a}(\alpha_n)\neq0$ along this line.
\end{proof}

Let us discuss the implications of Lemma~\ref{lem:whitham kernel node} for transversality to the set of nodal curves.
Near a nodal curve $a \in \calU^g_n \subset \calU^g$, which includes the case $a \in \calH^g_n$, for each node $\alpha_j$ the polynomials that keep a double root near $\alpha_j$ form a hypersurface $\calN_j$.
On one side of $\calN_j$ the two roots near $\alpha_j$ form a conjugate pair off $\bbS^1$, and on the other side they are two roots on $\bbS^1$, compare Figure~\ref{fig:chord}.
Its tangent space is $T_a\calN_j=\{\dot{a}\mid\dot{a}(\alpha_j)=0\}$.
This can be seen by differentiating $a_t(\alpha_t)=0$ along a path in $\calN_j$ with double root $\alpha_t$; one obtains $\dot{a}(\alpha_j)=0$, because $a'(\alpha_j)=0$.
For $j$ with $D_j\neq0$, equation~\eqref{eq:node dot a} shows that a vector of $\ker \dd\Phi$ is tangent to $\calN_j$ if and only if $Q(\alpha_j)=0$.
Suppose first that $D_j\neq0$ for $j=1,\dots,n$.
By Lemma~\ref{lem:whitham kernel node}(ii), $(\dot{a},\dot{b}_1,\dot{b}_2)\mapsto Q$ is an isomorphism $\ker \dd\Phi\to\calP^2_\bbR$, under which $\ker \dd\Phi\cap T_a\calN_j$ corresponds to the plane $\{Q\in\calP^2_\bbR\mid Q(\alpha_j)=0\}$.
So $\ker \dd\Phi$ is transverse to every $\calN_j$.
If instead $D_n=0$, as in Lemma~\ref{lem:whitham kernel node}(iii), then $(\dot{a},\dot{b}_1,\dot{b}_2)\mapsto Q$ maps $\ker \dd\Phi$ onto $\{Q\in\calP^2_\bbR\mid Q(\alpha_n)=0\}$.
For $j<n$ this plane is not $\{Q\in\calP^2_\bbR\mid Q(\alpha_j)=0\}$, so $\ker \dd\Phi\not\subset T_a\calN_j$; and for $j=n$ Lemma~\ref{lem:whitham kernel node}(iii) gives a vector with $\dot{a}(\alpha_n)\neq0$.
So in both cases $\ker \dd\Phi$ is transverse to every $\calN_j$.

Along a path in $\calN_1\cap\dots\cap\calN_n$, that is, a path that keeps all $n$ double roots on $\bbS^1$, the factorizations $a = p^2\tilde{a}$ and $b_k = p\tilde{b}_k$ persist, with $p$ now varying through the positions of the double roots.
Hence $\dot{a} = 2p\dot{p}\tilde{a} + p^2\dot{\tilde{a}}$ and $\dot{b}_k = \dot{p}\tilde{b}_k + p\dot{\tilde{b}}_k$.
Since $\dot{a}(\alpha_j) = 0$, equation~\eqref{eq:whitham eval node} gives $c_k(\alpha_j) = 0$ for every $j$, so $c_k = p\tilde{c}_k$ with $\tilde{c}_k \in \calP^{\hat{g}+1}_\bbR$.
Substituting into~\eqref{eq:whitham ab}, a factor of $p^3$ cancels from both sides and we obtain
\[
2\tilde{a}\dot{\tilde{b}}_k - \dot{\tilde{a}}\tilde{b}_k = \ci\big(2\lambda\tilde{a}\tilde{c}_k' - (\lambda\tilde{a}' + \tilde{a})\tilde{c}_k\big),
\]
and similarly~\eqref{eq:whitham Q} becomes $\tilde{b}_2\tilde{c}_1 - \tilde{b}_1\tilde{c}_2 = \tilde{a}Q$.
These are the Whitham equations for $(\tilde{a},\tilde{b}_1,\tilde{b}_2)$, with the same $Q$.
After the rotation~\eqref{eq:normalization}, they become the Whitham equations~\eqref{eq:whitham ab} and~\eqref{eq:whitham Q} for $(\hat{a},\hat{b}_1,\hat{b}_2)$ on the normalization.
This explains intuitively why the nodal $B$-periods are `invisible' to the Whitham equations.

\begin{theorem}
\label{thm:Whitham deformation node}
Let $g\geq2$ and let $a\in\calU^g_n$ with $1\leq n\leq g$; that is, $a$ has exactly $n$ double roots $\alpha_1,\dots,\alpha_n\in\bbS^1$ and otherwise simple roots.
Further suppose that each $\alpha_j$ is a common root of $\calB_a$ of first order, that $\calB_a$ has no other common roots, and that $D_j\neq0$ for $j=1,\dots,n-1$.
\begin{enumerate}
\item
The isoperiodic set $\calU^g(a)$ is locally at $a$ a three-dimensional submanifold of $\calU^g$, transverse to every $\calN_j$.
\item
The isoperiodic set $\calU^g(a,\tau)$ is locally at $a$ a one-dimensional submanifold of $\calU^g$, transverse to $\calN_n$.
For each $j=1,\dots,n-1$, it is transverse to $\calN_j$ if and only if $D_n\neq0$.
\end{enumerate}
\end{theorem}

\begin{proof}
\emph{(i).}
By Lemma~\ref{lem:whitham kernel node}(ii) if $D_n \neq 0$, or by Lemma~\ref{lem:whitham kernel node}(iii) if $D_n = 0$, the kernel $\ker \dd\Phi$ is three-dimensional.
By the constant rank theorem, the isoperiodic set of $(a,b_1,b_2)$ is locally a three-dimensional submanifold of $\calB\oplus\calB$.
By the same argument as Theorem~\refinv{thm:Whitham deformation}{2.3}, with Lemma~\ref{lem:b_omega smooth node} in place of Lemma~\refinv{lem:differential period}{1.5}, the projection $\pi$ from this isoperiodic set to $\calU^g$ is a proper injective immersion.
Hence $\calU^g(a)$ is a three-dimensional submanifold.
Its tangent space $\dd\pi(\ker \dd\Phi)$ is transverse to every $\calN_j$, as shown in the discussion before the theorem.

\emph{(ii).}
As in the proof of Theorem~\refinv{thm:Whitham deformation}{2.3}, $\calU^g(a,\tau)$ is the subset of $\calU^g(a)$ on which $\tau=b_2(0)/b_1(0)$ is constant.
The formula~\eqref{eq:dot tau Q} for $\dot{\tau}$ only uses~\eqref{eq:whitham ab} and~\eqref{eq:whitham Q} at $\lambda=0$.
These follow from~\eqref{eq:whitham ab node} and~\eqref{eq:whitham Q node} on multiplying by $p$, so by Lemma~\ref{lem:whitham kernel node}(i) the formula holds on $T_a\calU^g(a)$.
If $D_n\neq0$, then $Q$ ranges over $\calP^2_\bbR$, so $Q(0)$ ranges over $\bbC$, and the kernel of $\dd\tau$ is $\{Q\in\bbR\lambda\}$, by the reality condition.
For nonzero $Q\in\bbR\lambda$ we have $Q(\alpha_j)\neq0$, and hence $\dot{a}(\alpha_j)\neq0$ by~\eqref{eq:node dot a}, for every $j$.
Thus $\calU^g(a,\tau)$ is not tangent to any $\calN_j$, which for a curve is the same as being transverse.
If $D_n=0$, then $Q$ ranges over $\{Q\in\calP^2_\bbR\mid Q(\alpha_n)=0\}$.
If moreover $Q(0)=0$, then $Q=Q'(0)\lambda$ with $Q'(0)\alpha_n=0$, so $Q=0$.
Hence $Q\mapsto Q(0)$ maps $\{Q\in\calP^2_\bbR\mid Q(\alpha_n)=0\}$ isomorphically onto $\bbC$, and the kernel of $\dd\tau$ is the kernel of $(\dot{a},\dot{b}_1,\dot{b}_2)\mapsto Q$, which by Lemma~\ref{lem:whitham kernel node}(iii) is spanned by a vector with $\dot{a}(\alpha_n)\neq0$.
So $\calU^g(a,\tau)$ is not tangent to $\calN_n$.
Since this vector has $Q=0$ and $D_j\neq0$ for $j<n$, equation~\eqref{eq:node dot a} gives $\dot{a}(\alpha_j)=0$, so $\calU^g(a,\tau)$ is tangent to $\calN_j$ for every $j<n$.
In both cases $\dd\tau$ is surjective on $T_a\calU^g(a)$, so $\calU^g(a,\tau)$ is a one-dimensional submanifold, with the stated transversality.
\end{proof}

We now turn our attention from whether the nodes open to how they open.
Assume the hypotheses of Lemma~\ref{lem:whitham kernel node} with $D_j\neq0$ for $j=1,\dots,n$, and let $(\dot{a},\dot{b}_1,\dot{b}_2)\in\ker \dd\Phi$.
This direction opens every node if and only if $Q(\alpha_j)\neq0$ for all $j$, which holds outside finitely many planes in $\calP^2_\bbR$.
A double root may split into a conjugate pair off $\bbS^1$ or two simple roots on $\bbS^1$.
For $n\leq3$, the at most three planes $\{Q\in\calP^2_\bbR\mid Q(\alpha_j)=0\}$ are in general position in the $3$-dimensional space $\calP^2_\bbR$, and so there are directions that split all double roots into conjugate pairs.
For $n\geq4$ this is not true in general.

We can describe these directions more concretely.
Since $\alpha_j$ is a local minimum of the real function $\lambda^{-g}a$ on $\bbS^1$, the sign of $\alpha_j^{-g}\dot{a}(\alpha_j)$ decides how the node $\alpha_j$ opens to first order: if it is positive then the double root splits into a conjugate pair off $\bbS^1$, and if it is negative then it splits into two simple roots on $\bbS^1$.
For each node $\alpha_j$, we may rewrite~\eqref{eq:node dot a} as
\begin{equation}
\label{eq:node opening}
\alpha_j^{-g}\dot{a}(\alpha_j)=\sigma_j\,\alpha_j^{-1}Q(\alpha_j),
\end{equation}
where we fix a basis with $b_1'(\alpha_j)\neq0$ for all $j=1,\dots,n$.
Such a basis exists because each $\alpha_j$ is a first-order common root of $\calB_a$, so $b\mapsto b'(\alpha_j)$ vanishes only on a proper subspace of $\calB_a$.
Then
\begin{align*}
\sigma_j
=\frac{2\ci\alpha_j^{2-g}a''(\alpha_j)^2}{D_j}
=\frac{\ci\alpha_j^{2-g}a''(\alpha_j)^2}{b_1'(\alpha_j)^2f'(\alpha_j)}
=\frac{\big(-\alpha_j^{2-g}a''(\alpha_j)\big)^2}{\big(-\alpha_j^{1-g}b_1'(\alpha_j)^2\big)\,\ci\alpha_jf'(\alpha_j)}
=\frac{|a''(\alpha_j)|^2}{|b_1'(\alpha_j)|^2\,\ci\alpha_jf'(\alpha_j)}.
\end{align*}
The second equality is $D_j=2b_1'(\alpha_j)^2f'(\alpha_j)$, from the discussion after Lemma~\ref{lem:whitham kernel node}.
For the last, write $\lambda = \alpha_je^{\ci\phi}$.
Then $-\alpha_j^{2-g}a''(\alpha_j)$ is the second derivative in $\phi$ at $\phi=0$ of the real function $\lambda^{-g}a$, which is positive since $\alpha_j$ is a nondegenerate local minimum.
Similarly $-\alpha_j^{1-g}b_1'(\alpha_j)^2$ is the square of the first derivative in $\phi$ of the real function $\lambda^{-(g+1)/2}b_1$ at its root $\phi=0$, which is positive.
The factor $\ci\alpha_jf'(\alpha_j)$ is real, since $f$ is real along $\bbS^1$.
Hence $\sigma_j$ is a nonzero real number.
Changing the basis $(b_1,b_2)$ multiplies $D_j$ and $Q$ by its determinant, and so can only preserve or reverse all the signs of $\sigma_j$ simultaneously.
Recall that the condition $D_j\neq0$ is that $\dd f$ does not vanish at $\alpha_j$.
Consequently, all $\sigma_j$ have the same sign if $\dd f$ has no roots on $\bbS^1$.

Hence the double root $\alpha_j$ splits into a conjugate pair off $\bbS^1$ if and only if $\alpha_j^{-1}Q(\alpha_j)$ has the sign of $\sigma_j$.
The pattern of signs of $\alpha_j^{-1}Q(\alpha_j)$, $j=1,\dots,n$, is rather rigid.
The function $\lambda^{-1}Q$ has up to $2$ roots on $\bbS^1$, so it partitions $\bbS^1$ into two open arcs (possibly empty): one on which it is positive and one on which it is negative.
Ordering the nodes $\alpha_j$ cyclically, the signs are partitioned into two (possibly empty) blocks.
Therefore if the signs of $\sigma_j$, which depend only on $(a,b_1,b_2)$, are not an arc of positives followed by an arc of negatives, there is no $Q$ that splits all double roots into conjugate pairs.
For $a \in \calH^g_n$, this means we cannot reach $\calH^g_0$ along $\calU^g(a)$, at least infinitesimally.
Along $\calU^g(a,\tau)$ the situation is even more rigid.
Its tangent has $Q\in\bbR\lambda$, so $\alpha_j^{-1}Q(\alpha_j)=Q'(0)$ for every $j$.
Therefore we can open $\calH^g_n$ into $\calH^g_0$ along $\calU^g(a,\tau)$ if and only if all $\sigma_j$ have the same sign.

\begin{corollary}
\label{cor:Willmore derivative node}
Let $a$ be as in Theorem~\ref{thm:Whitham deformation node}, and evaluate $\willmore$ in a frame with constant periods along $\calU^g(a,\tau)$, as in Theorem~\refinv{thm:Willmore derivative}{2.5}.
Then $\dot{\willmore}=0$ at $a$ if and only if $\frac{\dd}{\dd\lambda}\frac{b_2}{b_1}$ vanishes at $\alpha_n$.
\end{corollary}

\begin{proof}
The computation $\dot{\willmore}\operatorname{Area}=4Q'(0)$ of~\eqrefinv{eq:Willmore derivative}{2.10} only involves the coefficients of $\lambda^0$ and $\lambda^1$ in~\eqref{eq:whitham ab} and~\eqref{eq:whitham Q}.
These hold at $a$ by Lemma~\ref{lem:whitham kernel node}(i).
By the proof of Theorem~\ref{thm:Whitham deformation node}(ii), along $\calU^g(a,\tau)$ we know $Q\in\bbR\lambda$ if $D_n\neq0$, and $Q=0$ if $D_n=0$.
In the first case $Q'(0)\neq0$, since the tangent is nonzero and $(\dot{a},\dot{b}_1,\dot{b}_2)\mapsto Q$ is injective on $\ker \dd\Phi$ by Lemma~\ref{lem:whitham kernel node}(ii), so $\dot{\willmore}\operatorname{Area}=4Q'(0)\neq0$ by~\eqrefinv{eq:Willmore derivative}{2.10}.
In the second case $\dot{\willmore}=0$.
Finally, by the discussion after Lemma~\ref{lem:whitham kernel node}, $D_n=0$ if and only if $\frac{\dd}{\dd\lambda}\frac{b_2}{b_1}$ vanishes at $\alpha_n$.
\end{proof}

This is the extension of Theorem~\refinv{thm:Willmore derivative}{2.5} to nodal curves.
In the nonsingular case, $\dot{\willmore}$ vanishes along $\calH^g_0(a,\tau)$ exactly when $\calB_a$ has a common root on $\bbS^1$, that is, $a \in \calS^g$.
By continuity, $\dot{\willmore}$ vanishes on nodal curves that lie in the closure of $\calS^g$.
In contrast, $\calB_a$ is forced to have a common root at a node, yet $\dot{\willmore}$ does not have to vanish.
The common roots at nodes are in a sense ``freebies''.
The following lemma makes this precise: a nodal curve with $D_j\neq0$ at every node does not lie in the closure of $\calS^g$, consistent with Corollary~\ref{cor:Willmore derivative node}.
We expect a converse (under a suitable nondegeneracy condition), that a nodal curve with $D_n=0$ lies in the closure of $\calS^g$, and hope to consider this in future work.

\begin{lemma}
\label{lem:closure S}
Let $a\in\calU^g_n$ be such that each node $\alpha_j$ is a common root of $\calB_a$ of first order, $\calB_a$ has no other common roots, and $D_j\neq0$ for $j=1,\dots,n$.
Then no $\tilde{a}\in\calU^g_0$ near $a$ has a common root of $\calB_{\tilde{a}}$.
In particular, $a$ does not lie in the closure of $\calS^g$.
\end{lemma}
\begin{proof}
For each node $\alpha_j$, we express the condition that $\calB_{\tilde{a}}$ has a common root near $\alpha_j$ as the vanishing of a smooth function $\varphi_j$, and compute the derivative of $\varphi_j$ along a Whitham direction that opens the node.
Fix $j$ and a smooth frame $(b_1,b_2)$ of $\calB$ near $a$ with $b_1'(\alpha_j)\neq0$, which exists by Lemma~\ref{lem:b_omega smooth node}.
For $\tilde{a}\in\calU^g$ near $a$, the polynomial $b_1$ has a unique root $\beta$ near $\alpha_j$, which is a smooth function of $\tilde{a}$.
It lies on $\bbS^1$, since its reflection under $\lambda\mapsto\bar{\lambda}^{-1}$ is also a root of $b_1$ near $\alpha_j$.
Hence $\varphi_j=\beta^{-(g+1)/2}\,b_2(\beta)$, the evaluation of $b_2$ at the root of $b_1$, is a smooth real function near $a$ (for a continuous choice of the square root).
The condition to have a common root of $\calB_{\tilde{a}}$ near $\alpha_j$ is exactly $\varphi_j=0$.
In particular $\varphi_j$ vanishes on the hypersurface $\calN_j$, whose tangent space is $\{\dot{a}\mid\dot{a}(\alpha_j)=0\}$ by the discussion before Theorem~\ref{thm:Whitham deformation node}.

We compute therefore $\dot{\varphi}_j$ in a direction transverse to $\calN_j$.
Let $(\dot{a},\dot{b}_1,\dot{b}_2)$ be the vector constructed in the proof of Lemma~\ref{lem:whitham kernel node}(i), which solves the Whitham equations~\eqref{eq:whitham ab} with $c_k=u\tilde{b}_k$ and has $\dot{a}(\alpha_j)\neq0$.
As $\dot{\Theta}_{b_k}$ is exact, the $A$-periods are preserved, so $\dot{b}_k$ differs from the derivative of the frame along $\dot{a}$ by an element of $\calB_a$.
Such an element vanishes at $\alpha_j$, so $\dot{\beta}=-\dot{b}_1(\alpha_j)/b_1'(\alpha_j)$.
Recall from the proof of Lemma~\ref{lem:whitham kernel node} that antisymmetrizing~\eqref{eq:whitham deriv eval node} gives $\tilde{b}_2\dot{b}_1-\tilde{b}_1\dot{b}_2=\ci\alpha_j(c_2\tilde{b}_1'-c_1\tilde{b}_2')$ at $\alpha_j$, using only the Whitham equations.
Therefore, since $b_2(\alpha_j)=0$ and evaluating everything at $\alpha_j$,
\begin{align*}
\alpha_j^{(g+1)/2}\dot{\varphi}_j
= \dot{b}_2+b_2'\dot{\beta}
= \frac{\dot{b}_2b_1'-\dot{b}_1b_2'}{b_1'}
= -\frac{\ci\alpha_jp'(c_2\tilde{b}_1'-c_1\tilde{b}_2')}{b_1'}
= \frac{\dot{a}\,D_j}{2a''\,b_1'}.
\end{align*}
This calculation also holds at every point of $\calN_j$ near $a$, since the hypotheses are open conditions.

As $D_j\neq0$ near $a$ for $j=1,\dots,n$, this formula shows that $\dd\varphi_j$ has rank $1$ at every point of $\calN_j$ near $a$.
By the constant rank theorem, near $a$ the zero set of $\varphi_j$ is a submanifold of codimension $1$.
It contains the hypersurface $\calN_j$, and so the two coincide near $a$.
Thus for $\tilde{a}\in\calU^g_0$ near $a$, the plane $\calB_{\tilde{a}}$ has no common root near any node.
Since $\calB$ is continuous and $\calB_a$ has no common roots other than the nodes, it has no common root elsewhere either.
As $\calS^g\subset\calU^g_0$, in particular $a$ does not lie in the closure of $\calS^g$.
\end{proof}

\section{Isoperiodic Invariants at Nodal Curves}
\label{sec:invariants nodes}

We recall the discrete invariant of~\cite[Section~3]{Hasse2026a}.
A change of canonical basis acts on $\BPer_a[\calB_a]\subset\bbR^g$ through a matrix $X\in\GL(g,\bbZ)$~\eqrefinv{eq:symplectic transformation}{3.1}.
By Lemma~\refinv{lem:CA adapted basis}{3.2}, two adapted canonical bases differ by $X\in\CA(g)=\pi_2^{-1}[\Perm(g,\bbF_2)]$, where $\pi_2$ is reduction modulo $2$.
In the doubly periodic case, $\Pi=\BPer_a[\calB_a]\in\Gr(2,\bbQ^g)$, so that $\Lambda=\Pi\cap\bbZ^g$ is a lattice of rank $2$.
Theorem~\refinv{thm:CAg orbits of Pi}{3.4} states that the $\CA(g)$-orbits of such $\Pi$ correspond to the triples $(n_1,n_2,n_3)$ that tally the following:
for a basis $v_1,v_2$ of $\Lambda$,
how often each nonzero value of $\bbF_2^2$ occurs among the binary coordinate vectors $(\pi_2(v_{1,i}),\pi_2(v_{2,i}))\in\bbF_2^2$ for $i=1,\dots,g$.
The tally is independent of the basis $v_1,v_2$, although a change of basis may permute the nonzero values of $\bbF_2^2$.
Because the binary coordinate vectors span $\bbF_2^2$, at least two counts are nonzero.
Listing the counts from largest to smallest, we have the conditions $n_1\geq n_2\geq n_3\geq0$, $n_2\geq1$ and $n_1+n_2+n_3\leq g$.
By definition, $\BPer_a[\calB_a]$, and hence the triple, is constant along isoperiodic sets; it is an invariant of $\calH^g_0(a)$.
The definition uses only $\BPer_a$ on $\calB_a$ and adapted canonical bases.
By Lemma~\ref{lem:b_omega smooth node} and the bases of homology of Section~\ref{sec:extending singular}, it therefore extends to every $a\in\calU^g$ in the doubly periodic case, for $g\geq2$, and the triple is constant on $\calU^g(a)$.

Our goal for this section is to show that every triple that meets the three inequalities of the previous paragraph actually occurs.
Lemma~\ref{lem:invariants normalization} relates the lattice of a nodal curve to that of its normalization.
Lemma~\ref{lem:insert node} builds on that, showing how inserting one node into a doubly periodic curve changes the triple.
Theorem~\ref{thm:all triples} combines this with Theorem~\ref{thm:Whitham deformation node} to prove it.
We close with a short discussion of $\tau$, the namesake invariant of $\calU^g(a,\tau)$.

\begin{lemma}
\label{lem:invariants normalization}
Let $a\in\calU^g_n$ with $g-n\geq2$, and let $\hat{a}\in\calU^{g-n}_0$ be its normalization~\eqref{eq:normalization}.
Choose a basis of homology as in Section~\ref{sec:extending singular}, so that the first $n$ $A$-cycles are the neck cycles, and let $\mathrm{pr}\colon\bbR^g\to\bbR^{g-n}$ forget the first $n$ coordinates.
\begin{enumerate}
\item
The projection $\mathrm{pr}$ restricts to a complex linear isomorphism from $\Pi=\BPer_a[\calB_a]$ onto $\hat{\Pi}=\BPer_{\hat{a}}[\calB_{\hat{a}}]$.
\item
If $a$ is in the doubly periodic case, then so is $\hat{a}$, and $\mathrm{pr}[\Lambda]$ is a sublattice of finite index in $\hat{\Lambda}=\hat{\Pi}\cap\bbZ^{g-n}$.
\item
If $\hat{a}$ is in the doubly periodic case, then $a$ is in the doubly periodic case if and only if the nodal $B$-periods of the elements of $\calB_a$ corresponding to a basis of $\hat{\Lambda}$ are rational.
\end{enumerate}
\end{lemma}

\begin{proof}
The last $g-n$ cycles of the basis avoid the nodes, and they pull back to an adapted canonical basis of homology of $\Sigma_{\hat{a}}$.
Recall the isomorphism $\calB_a\to\calB_{\hat{a}}$, $b\mapsto\hat{b}$, described after~\eqref{eq:normalization}: it is complex linear, and $\Theta_{\hat{b}}$ is the pullback of $\Theta_b$ to the normalization, up to rotation.
Hence $\BPer_{\hat{a}}(\hat{b})=\mathrm{pr}(\BPer_a(b))$.
As $g-n\geq2$, the map $\BPer_{\hat{a}}$ is injective on $\calB_{\hat{a}}$ by Lemma~\ref{lem:b_omega smooth node}, and so $\mathrm{pr}$ restricts to an isomorphism $\Pi\to\hat{\Pi}$.
It is complex linear because $b\mapsto\hat{b}$ is, which proves (i).

Recall that a plane $\Pi\subset\bbR^g$ meets $\bbZ^g$ in a lattice of rank $2$ if and only if it is rational, that is, $\Pi\in\Gr(2,\bbQ^g)$; so $a$ is in the doubly periodic case if and only if $\Pi$ is rational.
If $\Pi$ is rational then so is $\hat{\Pi}=\mathrm{pr}[\Pi]$, and $\mathrm{pr}[\Lambda]\subset\hat{\Lambda}$ is a subgroup of rank $2$ of a lattice of rank $2$, which proves (ii).

For (iii), let $\hat{w}_1,\hat{w}_2$ be a basis of $\hat{\Lambda}$, and let $w_k\in\Pi$ be their preimages under $\mathrm{pr}$.
The first $n$ coordinates of $w_k$ are the nodal $B$-periods $-\ci\int_{\alpha_j^-}^{\alpha_j^+}\Theta_{b_k}$ of the corresponding $b_k\in\calB_a$, and the remaining coordinates are integers.
If the nodal $B$-periods are rational, then $\Pi$ is spanned by the rational vectors $w_1,w_2$.
Conversely, if $\Pi$ is rational then the inverse of $\mathrm{pr}|_\Pi$ is defined over $\bbQ$, and so $w_1,w_2$ are rational.
\end{proof}

As a matter of cosmic history, it has always been easier to destroy than to create~\cite{Meyer82}; one must insert nodes carefully to maintain double periodicity.
By Lemma~\ref{lem:invariants normalization}(iii), we need to arrange that the nodal $B$-periods are rational.
With one node they form a vector in $\bbR^2$, and the following lemma shows that they move freely as the node slides along $\bbS^1$ and the normalization moves along its $1$-dimensional isoperiodic set.

\begin{lemma}
\label{lem:insert node}
Let $\hat{g}\geq2$, and let $\hat{a}\in\calH^{\hat{g}}_0$ be in the doubly periodic case, with triple $(\hat{n}_1,\hat{n}_2,\hat{n}_3)$, such that $\calB_{\hat{a}}$ has no common roots.
Let $(n_1,n_2,n_3)$ be either $(\hat{n}_1,\hat{n}_2,\hat{n}_3)$ or obtained from it by increasing one entry by one and reordering.
Then there is $a\in\calH^{\hat{g}+1}_1$ in the doubly periodic case with triple $(n_1,n_2,n_3)$, whose node is a common root of $\calB_a$ of first order with $D_1\neq0$, and such that $\calB_a$ has no other common roots.
\end{lemma}

\begin{proof}
We construct a two-parameter family of curves with one node, show that the nodal $B$-periods are local coordinates on the parameter space, and choose the parameters so that the periods are rational with odd denominator.
Let $\hat{a}_t$, $t\in(-1,1)$, parametrize $\calH^{\hat{g}}_0(\hat{a},\hat{\tau})$ near $\hat{a}_0=\hat{a}$, and let $(\hat{b}_1,\hat{b}_2)$ be a frame of $\calB$ along it with constant periods, such that $\BPer_{\hat{a}}(\hat{b}_1),\BPer_{\hat{a}}(\hat{b}_2)$ is a basis $\hat{w}_1,\hat{w}_2$ of $\hat{\Lambda}$.
For $\alpha\in\bbS^1$, let $a_{\alpha,t}\in\calH^{\hat{g}+1}_1$ have a node at $\alpha$ and normalization $\hat{a}_t$; up to a rotation, which changes none of the quantities below, $a_{\alpha,t}=p_\alpha^2\hat{a}_t$.
Consider the nodal $B$-periods $N_k(\alpha,t)=-\ci\int_{\alpha^-}^{\alpha^+}\Theta_{\hat{b}_k}$, $k=1,2$, where $\alpha^\pm$ are the two points over $\alpha$.
Writing $\alpha=e^{\ci\phi}$ and using $\sigma^\ast\Theta_{\hat{b}_k}=-\Theta_{\hat{b}_k}$ at the two endpoints, we have $\partial_\phi N_k=-2\ci\cdot\ci\alpha\,\hat{b}_k(\alpha)/(\alpha\hat{\nu}(\alpha^+))=2\hat{b}_k(\alpha)/\hat{\nu}(\alpha^+)$.
Along $\hat{a}_t$ we have $\dot{\Theta}_{\hat{b}_k}=\dd(\ci\hat{c}_k/\hat{\nu})$ with $\hat{b}_2\hat{c}_1-\hat{b}_1\hat{c}_2=\hat{a}\hat{Q}$ and $\hat{Q}\in\bbR\lambda\setminus\{0\}$, by the description of the tangent line of $\calH^{\hat{g}}_0(\hat{a},\hat{\tau})$ in Section~\ref{sec:background}.
Hence $\partial_tN_k=2\hat{c}_k(\alpha)/\hat{\nu}(\alpha^+)$.
The Jacobian determinant of $(N_1,N_2)$ in $(\phi,t)$ is therefore
\[
\frac{4\big(\hat{b}_1\hat{c}_2-\hat{b}_2\hat{c}_1\big)(\alpha)}{\hat{\nu}(\alpha^+)^2}
= -\frac{4\,\hat{a}(\alpha)\hat{Q}(\alpha)}{\hat{\nu}(\alpha^+)^2}
\neq0,
\]
since $\hat{a}$ has no roots on $\bbS^1$ and a nonzero $\hat{Q}\in\bbR\lambda$ does not vanish on $\bbS^1$.
So $(N_1,N_2)$ is a local diffeomorphism near every $(\alpha,0)$.

Fix $\alpha_0\in\bbS^1$ at which $D_1\neq0$; this excludes only the finitely many roots on $\bbS^1$ of $\hat{b}_1\hat{b}_2'-\hat{b}_2\hat{b}_1'$, compare the commentary after Lemma~\ref{lem:whitham kernel node}.
The image of a small neighborhood of $(\alpha_0,0)$ is open, so it contains a point $(r_1/s,r_2/s)$ with $r_1,r_2\in\bbZ$, $s$ odd, and $(r_1,r_2)$ of any prescribed parity $\varepsilon\in\bbF_2^2$.
Let $a$ be the corresponding curve $a_{\alpha,t}$.
For $(\alpha,t)$ near $(\alpha_0,0)$ we still have $D_1\neq0$, and $\calB_{\hat{a}_t}$ still has no common roots.
Hence the node is a common root of $\calB_a$ of first order, and $\calB_a$ has no other common roots.
By Lemma~\ref{lem:invariants normalization}(iii), $a$ is in the doubly periodic case.

It remains to compute the triple.
An element $m_1\hat{w}_1+m_2\hat{w}_2$ of $\hat{\Lambda}$ lifts to an element of $\Pi$ whose nodal $B$-period is $(m_1r_1+m_2r_2)/s$, so $\mathrm{pr}[\Lambda]=\{m_1\hat{w}_1+m_2\hat{w}_2\mid m_1r_1+m_2r_2\equiv0\bmod s\}$.
This contains $s\hat{\Lambda}$, so its index in $\hat{\Lambda}$ is odd, and reduction modulo $2$ maps $\mathrm{pr}[\Lambda]$ onto $\pi_2[\hat{\Lambda}]$.
Since $\GL(2,\bbZ)\to\GL(2,\bbF_2)$ is surjective, we may choose a basis $m^1,m^2$ of $\mathrm{pr}[\Lambda]$, in the coordinates of $\hat{w}_1,\hat{w}_2$, that reduces to the standard basis modulo $2$.
Let $v_1,v_2\in\Lambda$ be its lifts.
Their last $\hat{g}$ binary coordinate vectors agree with those of $\hat{w}_1,\hat{w}_2$, and the first is $\pi_2\big((m^k_1r_1+m^k_2r_2)/s\big)_{k=1,2}=\pi_2(r_1,r_2)=\varepsilon$, as $s$ is odd.
So the tally of $\Lambda$ is that of $\hat{\Lambda}$ with one more occurrence of $\varepsilon$.
Choosing $\varepsilon=0$, or the value whose tally is to be increased, gives the triple $(n_1,n_2,n_3)$.
\end{proof}

\begin{theorem}
\label{thm:all triples}
Let $g\geq2$.
Every triple $(n_1,n_2,n_3)$ of integers with $n_1\geq n_2\geq n_3\geq0$, $n_2\geq1$ and $n_1+n_2+n_3\leq g$ is the triple of some $a\in\calH^g_0$ in the doubly periodic case.
\end{theorem}

\begin{proof}
We induct on $g$, producing examples for which $\calB_a$ has no common roots.
For $g=2$, every $a\in\calH^2_0$ is in the doubly periodic case, since $\Pi=\bbR^2$, and its triple is $(1,1,0)$, which is the only permissible triple.
For $a\in\calH^2_0$, the plane $\calB_a$ has at most one common root, which therefore lies on $\bbS^1$, so the curves for which $\calB_a$ has a common root form the hypersurface $\calS^2$ of $\calH^2_0$~\cite[Lemma~1.1]{Carberry2023arxiv}.
We may choose $a$ outside it.

Suppose the claim holds for $g$, and let $(n_1,n_2,n_3)$ be a permissible triple for $g+1$.
If $n_1+n_2+n_3\leq g$, it is permissible for $g$.
Otherwise $n_1+n_2+n_3=g+1\geq3$, and decreasing $n_1$ by one and reordering gives a permissible triple for $g$, from which $(n_1,n_2,n_3)$ is obtained by increasing one entry by one.
In either case, by the induction hypothesis and Lemma~\ref{lem:insert node}, there is $a\in\calH^{g+1}_1$ in the doubly periodic case with triple $(n_1,n_2,n_3)$, satisfying the hypotheses of Theorem~\ref{thm:Whitham deformation node} with $n=1$ and $D_1\neq0$.
By that theorem, $\calU^{g+1}(a)$ is a submanifold transverse to $\calN_1$, so it contains points $\tilde{a}$ arbitrarily close to $a$ on the side of $\calN_1$ where the double root splits into a conjugate pair, and these lie in $\calH^{g+1}_0$.
As $\tilde{a}\in\calU^{g+1}(a)$, it has the same triple as $a$, namely $(n_1,n_2,n_3)$ as desired.
By Lemma~\ref{lem:closure S}, $\calB_{\tilde{a}}$ has no common roots for $\tilde{a}$ close enough to $a$, so $\tilde{a}$ can serve in the next step of the induction.
\end{proof}

Recall the continuous invariant $\tau=\varphi(w_2)/\varphi(w_1)$ of~\eqrefinv{eq:def tau}{3.3}, for a basis $w_1,w_2$ of $\Lambda$ and a complex linear isomorphism $\varphi\colon(\Pi,J)\to\bbC$.
By Theorem~\refinv{thm:CAg orbits of τ}{3.6}, up to the stabilizer of $\Pi$ in $\CA(g)$, it takes values in a fundamental domain $\calF_1$, $\calF_3$ or $\calF_6$, according to whether three, two or none of the $n_i$ coincide, and it is constant on $\calH^g_0(a,\tau)$.
Like the triple, it is defined purely in terms of the data $\BPer_a[\calB_a]$ and $J$, and so it extends to $\calU^g$ in the doubly periodic case.
By Lemma~\ref{lem:invariants normalization}(i), $\mathrm{pr}$ is a complex linear isomorphism $\Pi\to\hat{\Pi}$ that maps $\Lambda$ into $\hat{\Lambda}$.
So $\tau$ of $\Lambda$ is determined by $\hat{\tau}$ of $\hat{\Lambda}$ and the position of the sublattice $\mathrm{pr}[\Lambda]$ in $\hat{\Lambda}$.
For one node, as in the proof of Lemma~\ref{lem:insert node}, $\hat{\Lambda}/\mathrm{pr}[\Lambda]$ is isomorphic to the image of $(m_1,m_2)\mapsto m_1r_1+m_2r_2$ in $\bbZ/s\bbZ$, a cyclic group of some odd order $m$ dividing $s$.
Then $\mathrm{pr}[\Lambda]$ is spanned by $\hat{w}_1$ and $m\hat{w}_2$ for a suitable basis of $\hat{\Lambda}$, and $\tau=m\hat{\tau}$.

Locally, $\tau$ can be changed arbitrarily within $\calU^g(a)$, including at nodal curves.
Under the hypotheses of Theorem~\ref{thm:Whitham deformation node}, the proof of its part (ii) shows, by~\eqref{eq:dot tau Q}, that $\dd\tau$ is surjective on $T_a\calU^g(a)$, since $Q(0)$ ranges over $\bbC$.
Hence for a fixed triple the set of values of $\tau$ that occur is open.

\bibliographystyle{alpha}
\bibliography{cmc_moduli_gg_g2.bib}

\bigskip
{\small
Ross Ogilvie, Universit\"at Mannheim\\
\emph{Email address:} \href{mailto:r.ogilvie@uni-mannheim.de}{\texttt{r.ogilvie@uni-mannheim.de}}
\par}
\end{document}